\documentclass[a4paper,11pt,reqno]{amsart}

\usepackage[top=3cm, bottom=3cm, left=3cm, right=3cm]{geometry}

\usepackage[utf8]{inputenx}

\usepackage[english]{babel}
\usepackage{microtype}

\usepackage{mathtools}
\usepackage{amssymb}
\usepackage{amsfonts}
\usepackage{amsthm}
\usepackage{mathrsfs}
\usepackage{stmaryrd}

\usepackage{graphicx}
\usepackage{tikz}
\usepackage{tikz-cd}
\usepackage{dynkin-diagrams}
\usepackage[all]{xy}

\usepackage{tensor}
\usepackage{centernot}
\usepackage{faktor}
\usepackage{xfrac}
\usepackage{esint}
\usepackage{bbm}

\usepackage{hyperref}

\usepackage{xcolor}
\usepackage{array, tabularx}
\usepackage{mathdots}
\usepackage{comment}
\usepackage{paralist}

\theoremstyle{plain}
\newtheorem{thm}{Theorem}[section]
\newtheorem{prop}[thm]{Proposition}
\newtheorem{lem}[thm]{Lemma}
\newtheorem{cor}[thm]{Corollary}
\newtheorem{conj}[thm]{Conjecture}

\theoremstyle{definition}
\newtheorem{defn}[thm]{Definition}
\newtheorem{ex}[thm]{Example}

\theoremstyle{remark}
\newtheorem{rem}[thm]{Remark}

\begin{document}
	\newpage
	
	\title{Hodge decomposition for Kato manifolds}

  \author{Giacomo Perri}

	\address{Giacomo Perri \newline
		\textsc{\indent Institut for Matematik, Aarhus University\newline 
			\indent 8000, Aarhus C, Denmark}}
	\email{g.perri@math.au.dk}

\thanks{The author is supported by the Sapere Aude project $\lq\lq$Conformal geometry: metrics and cohomology" at Aarhus University}

\begin{abstract}
  We prove that any Kato manifold satisfies the Hodge decomposition, in the sense that $b_k=\sum_{p+q=k}h^{p, q}$, by relating its cohomology to the corresponding cohomology of its modification data. We give, therefore, more evidence supporting a conjecture of Ornea--Verbitsky stating that compact locally conformally K\" ahler manifolds satisfy the Hodge decomposition. We further study Bott--Chern and Aeppli cohomology of Kato
manifolds, showing that in certain degrees they coincide with Dolbeault cohomology.
\end{abstract}

\maketitle


\section{Introduction}

It is a fundamental result in complex geometry that the presence of a Kähler metric on a compact complex manifold imposes a special cohomological behaviour. One of the most interesting cohomological restrictions that compact Kähler manifolds satisfy is the $dd^c$-lemma. This further implies that they satisfy properties such as Hodge decomposition 
\begin{equation}
    b_k=\sum_{p+q=k}h^{p,q},
\end{equation}
and symmetry of the Hodge numbers ($h^{p,q}=h^{q,p}$), which eventually lead to special features of the Betti numbers.

Generic compact complex manifolds do not satisfy special cohomological properties, apart from some very general ones, such as Fr\" olicher inequality ($b_k \leq \sum_{p+q=k}h^{p, q}$) or Serre duality ($h^{p, q}=h^{n-p, n-q}$). A fundamental question, central to the modern problem of classifying complex manifolds and geometric structures, is whether a special Hermitian metric imposes cohomological obstructions.
One type of metric that has been intensively studied in the non-K\"ahler setting is the {\textit{locally conformally K\" ahler}} (lcK) metric. LcK metrics provide numerous examples—almost all non-K\"ahler surfaces, O--T manifolds, and Kato manifolds—and also offer a link to K\"ahler geometry through a special subclass called Vaisman manifolds. In general, compact lcK manifolds are not known to have cohomological restrictions that arise directly from the metric, apart from having a non-vanishing Bott–Chern group \(H_{BC}^{1,1}\) and satisfying the inequality \(2h^{0,1} > b_1\), both consequences of Vaisman’s theorem \cite{Vaisman}. However, empirically, Hodge decomposition has appeared as a common denominator for classes of manifolds admitting lcK metrics, although the methods used to establish it vary significantly from case to case.

Indeed, if the compact lcK manifold is Vaisman, the result is known by the work of Tsukada \cite{Tsukada}, whose proof relies heavily on the canonical transversal K\" ahler foliation that
Vaisman manifolds possess. Compact lcK surfaces also satisfy this property for algebraic reasons related to their low dimension. LcK manifolds with potential satisfy it due to complex deformation arguments \cite[Corollary 31.28]{ornea2024principleslocallyconformallykahler}, and finally, for Oeljeklaus--Toma manifolds, the result was proven in \cite{OTHodge} and is of a complex analytic nature. Only one class of known explicit examples of compact lcK manifolds remains for which the Dolbeault cohomology has not yet been computed—and therefore the property is not known to hold—namely Kato manifolds, as shown in \cite{istrati2019classkatomanifolds} and \cite{IOPR22}.  
So far, all known lcK manifolds either fall into one of the cases mentioned above or can be obtained from compact lcK manifolds via blow-ups~\cite{orneablowupslcK} or modifications~\cite{barbaro2025calabiyaulocallyconformallykahler}.
These blow-ups admitting an lcK structure are performed along specific 
complex submanifolds, which turn out to be Kähler. 
Hence, if the original lcK manifold satisfies the Hodge decomposition, 
so does its blow-up, by the blow-up formulas in Theorem~\ref{blowupd} and 
Theorem~\ref{blowupD}. 
The same conclusion holds when the modification is performed at a single point 
(Remark~\ref{lckremark}).
Motivated by this, the following conjecture has been formulated:

\begin{conj}[Ornea--Verbitsky]
    Any compact lcK manifold satisfies the Hodge decomposition.
\end{conj}

The current work aims at bringing more evidence supporting the conjecture by proving it on the full class of Kato manifolds:

\begin{thm}\label{Hodge decomposition}
    Any Kato manifold satisfies the Hodge decomposition.
\end{thm}

Kato manifolds are compact complex manifolds containing a global spherical shell. 
In \cite{Kato77}, it is proved that any such manifold can be obtained through an explicit construction. 
One starts with a proper holomorphic modification 
\(\pi : \widehat{\mathbb{B}} \to \mathbb{B}\) at finitely many points, 
together with a holomorphic embedding 
\(\sigma : \overline{\mathbb{B}} \hookrightarrow \widehat{\mathbb{B}}\). 
The Kato manifold \(X = X(\pi, \sigma)\) defined by \(\pi\) and \(\sigma\) 
is obtained by gluing the boundary \(\partial \widehat{\mathbb{B}}\) 
to the boundary of \(\sigma(\overline{\mathbb{B}})\) via the map \(\sigma \circ \pi\).  
Any Kato manifold can be modified via a small deformation to a modification \(\widehat{H}\) of a primary Hopf manifold $H$ \cite{Kato77}.  
Moreover, \(X\) is an lcK manifold if and only if \(\widehat{\mathbb{B}}\) is Kähler \cite[Theorem 10.3]{IOPR22}.

In \cite[Theorem~1.4]{IOPR22}, the authors obtain partial results for the Hodge decomposition of a class of toric Kato manifolds, using deformation arguments related to the toric structure. We note that our approach differs from that in \cite{IOPR22} and is more direct.  
It relies on expressing, via the Mayer--Vietoris sequence applied to a suitable open cover of a Kato manifold \(X\), the Dolbeault cohomology of \(X\) in terms of the Dolbeault cohomology of \(\widehat{\mathbb{B}}\) (Theorem~\ref{Hodge numbers}).  
We then observe that the Kato modification \(\pi\) induces a natural modification
\(\pi : \widehat{\mathbb{CP}^n} \to \mathbb{CP}^n\),
and, again through Mayer--Vietoris sequences, we relate the Hodge and Betti numbers of \(\widehat{\mathbb{B}}\) to those of \(\widehat{\mathbb{CP}^n}\) (Proposition~\ref{noWFT}).  
This allows us to express the Hodge and Betti numbers of \(X\) in terms of those of \(\widehat{\mathbb{CP}^n}\) (Theorem~\ref{Hodge numbers21}), and thus reduce the Hodge decomposition on \(X\) to that on \(\widehat{\mathbb{CP}^n}\).  
Since \(\widehat{\mathbb{CP}^n}\) is a Moishezon manifold, the Hodge decomposition holds on it \cite[Theorems~2.2.14 and~2.2.18]{Moishezon}, and the result therefore extends to any Kato manifold.  
We also prove that \(X\) and \(\widehat{H}\) have the same Hodge and Betti numbers (Theorem \ref{samenumbers}).

In Section \ref{SectionExamples}, we present explicit computations of the Hodge numbers for two relevant classes of Kato manifolds.
The first class consists of Kato manifolds whose modification is obtained through a finite sequence of blow-ups at points.
The second class is that of toric Kato manifolds \cite{IOPR22}, for which the modification is induced by a toric modification at the origin.

Finally, we study the Bott--Chern and Aeppli cohomology of Kato manifolds,
proving that in certain degrees these cohomologies coincide with Dolbeault cohomology.

We restrict to Kato manifolds of complex dimension \(n \ge 3\), since in complex dimension one all compact complex manifolds are Kähler, and the Hodge decomposition for compact complex surfaces is well-known \cite[Chapter 4.2]{VandeVen}.

\hfill

\noindent {\it Acknowledgements.} The author sincerely thanks Alexandra Otiman for her support, encouragement, and many valuable discussions and suggestions. The author is also grateful to Nicolina Istrati for her comments and remarks.

\section{Preliminaries}
In this section, we briefly recall the construction of Kato manifolds along 
with their main properties; for a complete treatment,
we refer the reader to \cite{Kato77}, \cite{Dlo84} and 
\cite{istrati2019classkatomanifolds}.
We conclude by providing formulas for the computation of the Dolbeault 
and de~Rham cohomology of blow-ups.

\subsection{Kato manifolds}

We start by recalling the definition of a proper modification:
\begin{defn}
        A morphism \(\pi : \widehat{X} \to X\) between two equidimensional complex spaces is called a 
\emph{proper modification} if it satisfies the following conditions:
\begin{enumerate}[(i)]
    \item \(\pi\) is proper and surjective;
    \item there exist nowhere dense analytic subsets \(\widehat{E} \subset \widehat{X}\) and \(E \subset X\) such that
    \[
    \pi : \widehat{X} \setminus \widehat{E} \longrightarrow X \setminus E
    \]
    is a biholomorphism, where \(\widehat{E} := \pi^{-1}(E)\) is called the \emph{exceptional space} of the modification.
\end{enumerate}
\end{defn}

We now proceed with the construction of Kato manifolds, highlighting their main properties.\\
Any Kato manifold is constructed starting from a proper holomorphic modification 
\(\pi : \widehat{\mathbb{B}} \to \mathbb{B} \subset \mathbb{C}^n\) at finitely many points 
and a holomorphic embedding \(\sigma : \overline{\mathbb{B}} \hookrightarrow \widehat{\mathbb{B}}\). 
We shall call \((\pi, \sigma)\) a \emph{Kato data} and 
\(\gamma := \pi \circ \sigma : \mathbb{B} \to \mathbb{B}\) the corresponding germ.
Let
\begin{align*}
    W &:= \widehat{\mathbb{B}} \setminus \sigma(\mathbb{B}),\\
\partial_+ W &:= \partial \widehat{\mathbb{B}}, \\
\partial_- W &:= \partial \sigma(\mathbb{B}),
\end{align*}
and define
\begin{align*}
    g : \partial_+ W &\longrightarrow \partial_- W,\\
    g &:= \sigma \circ \pi \big|_{\partial_+ W}.
\end{align*}
Clearly, $g$ extends to a biholomorphism between small neighborhoods of the two components 
of \(\partial W\), and we obtain the Kato manifold
\[
X(\pi, \sigma) := \overline{W} / \sim, 
\]
where  \( x \sim y\) if \(x \in \partial_+ W\) and \(y = g(x) \in \partial_- W\).
It follows directly from the construction that any Kato manifold is compact and that its universal cover is a \(\mathbb{Z}\)-cover.\\
Let \((\pi, \sigma)\) be a Kato data, let \(P \subset \mathbb{B}\) denote the finite set of points modified by \(\pi\), 
and let \(E \subset \widehat{\mathbb{B}}\) denote the exceptional set of \(\pi\). 
If $\sigma(P)$ does not intersect $E$, then the Kato manifold 
$M(\pi, \sigma)$ is a modification of a primary Hopf manifold 
\cite[Proposition~1]{Kato77}. In this case, Theorem~\ref{Hodge decomposition} 
is easier to prove, as the cohomology of primary Hopf manifolds is known 
\cite{mall1991cohomology}, and follows from the same arguments used in 
Section~\ref{finalsection}.
Thus, in all that follows, we will always assume that $\sigma(P) \cap E \neq \emptyset$ 
for a Kato data. In particular, any Kato manifold 
$\tilde{X} = \tilde{X}(\tilde{\pi}, \tilde{\sigma})$ is obtained as a modification at 
finitely many points of a Kato manifold $X = X(\pi, \sigma)$ with $\pi$ a proper 
holomorphic modification at the origin and $\sigma(0) \in E$; see 
\cite[Lemma~1.3,1.4]{istrati2019classkatomanifolds} and 
\cite[Lemme~1.5,2.7, Part~I]{Dlo84}. We can then assume that any Kato manifold is obtained by a proper holomorphic modification $\pi$ at the origin with 
$\sigma(0) \in E$. Otherwise, the same arguments apply verbatim to the general case. As a direct consequence of the Schwarz Lemma \cite[Chapter 1.4, Theorem~6]{Shabat1992}, 
the germ \(\gamma\) is a holomorphic contraction at the origin.

\subsection{Blow-up formulas}\label{blowformulas}

We conclude this section with two theorems on the de~Rham and Dolbeault 
cohomology of blow-ups of complex manifolds. A corollary of the second one 
will play a fundamental role in the proof of Theorem~\ref{Hodge decomposition}. 
We will use these theorems for the direct computation of the Hodge and 
Betti numbers of simple examples of Kato manifolds. 
We note that, in \cite{Voisin2002}, Theorem~\ref{blowupd} is stated for 
compact Kähler manifolds, but the same proof applies to general compact 
complex manifolds.

\begin{thm}[{\cite[Theorem~7.31]{Voisin2002}}]\label{blowupd} 
Let $X$ be a compact complex manifold with $\dim_{\mathbb{C}} X = n$ and 
$Z \subseteq X$ a closed complex submanifold of complex codimension $r \ge 2$. 
Suppose that $\pi : \widetilde{X} \to X$ is the blow-up of $X$ along $Z$. 
Then for any $0 \le k \le 2n$, there is an isomorphism
\[
H_{dR}^k(\widetilde{X})
\;\cong\;
H_{dR}^k(X)
\;\oplus\;
\left(
\bigoplus_{i=1}^{r-1}
H_{dR}^{k-2i}(Z)
\right).
\]
\end{thm}

\begin{thm}[{\cite[Theorem~1.2]{Meng20}}]\label{blowupD} Let $X$ be a complex manifold with $\dim_{\mathbb{C}} X = n$ and 
$Z \subseteq X$ a closed complex submanifold of complex codimension $r \ge 2$. 
Suppose that $\pi : \widetilde{X} \to X$ is the blow-up of $X$ along $Z$. 
Then for any $0 \le p, q \le n$, there is an isomorphism
\[
H_{\bar{\partial}}^{p,q}(\widetilde{X})
\;\cong\;
H_{\bar{\partial}}^{p,q}(X)
\;\oplus\;
\left(
\bigoplus_{i=1}^{r-1}
H_{\bar{\partial}}^{\,p-i,\,q-i}(Z)
\right).
\]
\end{thm}

As a direct application of the Weak Factorization Theorem \cite[Theorem 0.3.1]{AKMWo02}, we obtain:

\begin{cor}[{\cite[Corollary~1.4]{rao2018dolbeaultcohomologiesblowingcomplex}}]\label{p00q}
Let $X$ and $\widetilde{X}$ be two bimeromorphically equivalent $n$-dimensional 
compact complex manifolds. Then, for any $0 \le p, q \le n$, the following 
Dolbeault cohomology isomorphisms hold:
\begin{align*}
    H^{p,0}_{\overline{\partial}}(\widetilde{X}) &\cong 
H^{p,0}_{\overline{\partial}}(X),\\ 
H^{0,q}_{\overline{\partial}}(\widetilde{X}) &\cong 
H^{0,q}_{\overline{\partial}}(X).
\end{align*}
\end{cor}

\section{Hodge numbers of a Kato manifold}\label{picture}
In this section, we compute the Dolbeault cohomology of a Kato manifold 
\(X = X(\pi,\sigma)\) with proper modification \(\pi : \widehat{\mathbb{B}} \to \mathbb{B}\), 
in terms of the Dolbeault cohomology of \(\widehat{\mathbb{B}}\).
To do this, we cover \(X\) with an open cover 
\(\{U, V\}\) and apply the Mayer--Vietoris sequence for sheaf cohomology 
(see \cite[Chapter IV, Theorem 3.11]{Demaillyag}). The open set \(U\) is given by a global spherical shell, while \(V\) is
biholomorphic to \(\widehat{\mathbb{B}} \setminus \mathbb{D}\), where
\(\mathbb{D}\) is a disk.
The choice of the open sets \(U\) and \(V\) will be described in more detail later.\\

Let \(\mathcal{A}_I\) and \(\mathcal{A}_E\) be the two annuli representing 
the gluing part of the Kato manifold $X$ via the map \(\sigma\pi\). We can assume \[
\mathcal{A}_E = \pi^{-1}\big(\mathbb{B} \setminus \overline{\mathbb{B}}_{1-\epsilon}\big)
\quad \text{with} \quad 0 < \epsilon \ll 1,
\]
and then 
\[
\mathcal{A}_I = \sigma\big(\mathbb{B} \setminus \overline{\mathbb{B}}_{1-\epsilon}\big).
\] 
Let 
$$V=\widehat{\mathbb{B}}\setminus\overline{\big(\sigma(\mathbb{B})\cup \mathcal{A}_E\big)}\cong \widehat{\mathbb{B}}\setminus{\overline{\sigma(\mathbb{B})}},$$
and let \(U\) be a neighborhood of the annuli 
\(\mathcal{A}_I, \mathcal{A}_E\) in \(X\), given by a new annulus \(U\) 
as in Figure \ref{figure}. We choose \(U\) such that \(U_2\) has no intersection 
with the exceptional set \(\pi^{-1}(\{0\})\subset \widehat{\mathbb{B}}\) and \(U_1, U_2\) are disjoint. 
By construction, we have   
\begin{align*}
   U&=U_1\cup \overline{(\mathcal{A}_E\equiv \mathcal{A}_I)}\cup U_2,\\
U\cap V&=U_1\sqcup U_2.
\end{align*}
Moreover, we define \(\mathbb{B}_1\) and \(\mathbb{B}_2\) in 
\(\widehat{\mathbb{B}}\) by 
\begin{align*}
    \mathbb{B}_1&=U_1\sqcup \overline{\sigma(\mathbb{B})},\\
\mathbb{B}_2&=\sigma(\mathbb{B})\setminus \overline{\mathcal{A}_I}.
\end{align*}
By construction, we have 
\begin{equation}
    \begin{split}
        \widehat{\mathbb{B}} \setminus \overline{\mathcal{A}_E} &= \pi^{-1}(\mathbb{B}_{1-\epsilon}),\\
    \mathbb{B}_2 &= \sigma(\mathbb{B}_{1-\epsilon})
    \end{split}
\end{equation}
and hence 
\begin{align}\label{preimageB2}
    \sigma^{-1}(\mathbb{B}_2) = \mathbb{B}_{1-\epsilon}.
\end{align}

\begin{figure}[!t]
    \centering
\begin{tikzpicture}
    \draw (0,0) circle (5); 
    \draw (0,0) circle (4.7); 
    \draw (0,0) circle (4.2); 
    \draw (0,0) circle (2);
    \draw (0,0) circle (1.5);
    \draw (0,0) circle (1.2);
    \draw (0,0) circle (0.7);

    \fill[gray, opacity=0.6, even odd rule] (0,0) circle (5) (0,0) circle (4.7);
    \fill[gray, opacity=0.4, even odd rule] (0,0) circle (4.7) (0,0) circle (4.2);
    \fill[gray, opacity=0.2, even odd rule] (0,0) circle (4.2) (0,0) circle (2);
    \fill[gray, opacity=0.4, even odd rule] (0,0) circle (2) (0,0) circle (1.5) ;
    \fill[gray, opacity=0.6, even odd rule] (0,0) circle (1.5) (0,0) circle (1.2);
    \fill[gray, opacity=0.4, even odd rule] (0,0) circle (1.2) (0,0) circle (0.7);
    \fill[white] (0,0) circle (0.7);



    \node at (0,4.9) {$\mathcal{A}_E$};
    \node at (0,1.4) {$\mathcal{A}_I$};
    \node at (-4.45,0) {$U_2$};
    \node at (-1.75,0) {$U_1$};
    \node at (-0.6,0.2) {$\sigma\pi(U_2)$};
    \node [purple] at (2.8,0.2) {$V$};
    \node [blue] at (3.85,2.5) {$U$};
    \node [blue] at (1.2,0.9) {$U$};
    \node [red] at (0.3,-1.7) {$\mathbb{B}_1$};
    \node [red] at (0.6,-0.5) {$\mathbb{B}_2$};
    \node [red] at (-0.8,-0.8) {$\sigma(\mathbb{B})$};

\draw[purple, thick, dashed] (0:1.5) -- (0:4.7); 
\draw[blue, thick, dashed] (30:1.2) -- (30:2); 
\draw[blue, thick, dashed] (30:4.2) -- (30:5); 
\draw[red, thick, ->] (270:0) -- (270:2); 
\draw[red, thick, ->] (300:0) -- (300:1.2); 
\draw[red, thick, ->] (240:0) -- (240:1.5); 

\end{tikzpicture}
\caption{Construction on $\widehat{\mathbb{B}}$}
\label{figure}
\end{figure}
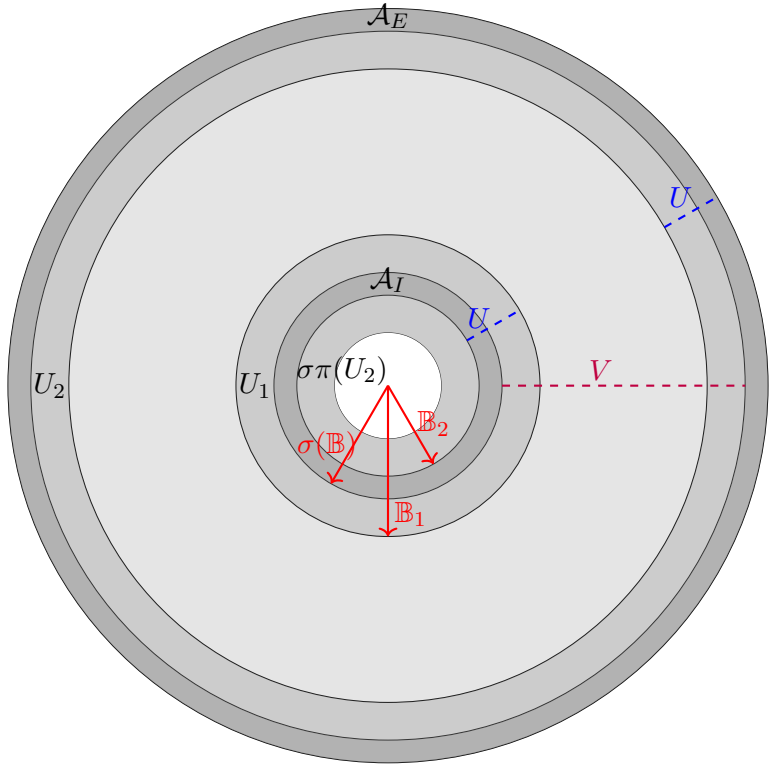

Before analyzing the Mayer--Vietoris sequence, we recall the well-known
Cartan's Theorem~B on the vanishing of sheaf cohomology for coherent analytic
sheaves on Stein manifolds; see \cite{CartanB} for reference.  
We then prove two technical lemmas. The first concerns the
sheaf cohomology of an annulus, while the second, which follows as a consequence
of the first, relates the cohomology of \(V\) to that of
\(\widehat{\mathbb{B}}\).

\begin{thm}[Cartan's Theorem B]
    Let $X$ be a Stein space and $\mathcal{F}$ a coherent analytic sheaf on $X$. 
Then for any $q > 0$ we have $H^{q}(X,\mathcal{F}) = 0$.
\end{thm}

\begin{lem}\label{annulus}
Let \(\mathcal{A}\) be an annulus in \(\mathbb{C}^n\), with \(n \ge 3\).
Then, for any \(p\), one has
\begin{equation*}
H^q(\mathcal{A}, \Omega^p) = 0 \quad \forall\, 1 \le q \le n-2.
\end{equation*}
\end{lem}

\begin{proof}
First, we assume that the annulus \(\mathcal{A}\) is given by 
\(\mathbb{B}_2 \setminus \overline{\mathbb{B}}\), where here $\mathbb{B}_2$ is the ball of radius two centered at the origin.
For any \(q \ge 1\) we have the isomorphism:
\[
H^q(\mathbb{B}_2 \setminus \overline{\mathbb{B}}, \Omega^p)
\cong
H^q(\mathbb{C}^n \setminus \overline{\mathbb{B}}, \Omega^p).
\]
This is a direct consequence of the Mayer--Vietoris sequence applied to
\[
\mathbb{C}^n = (\mathbb{C}^n \setminus \overline{\mathbb{B}}) \cup \mathbb{B}_2.
\]
Next, we observe that \(\mathbb{C}^n \setminus \overline{\mathbb{B}}\) is a \(1\)-concave manifold.
This follows, for instance, by considering the function
\(\rho(z) = |z|^2\),
which is strictly \(1\)-pseudoconvex on \(\mathbb{C}^n \setminus \overline{\mathbb{B}}\)
and takes values greater than \(1=\rho_{|_{\partial(\mathbb{C}^n \setminus \overline{\mathbb{B}})}}\) there.
By \cite[Theorem 11]{AG62}, we obtain
\begin{equation}
\dim H^q(\mathbb{C}^n \setminus \overline{\mathbb{B}}, \Omega^p) < +\infty
\quad \forall\, q < n-1.
\end{equation}
As proved in~\cite{Laufer}, since
\(\mathbb{C}^n \setminus \overline{\mathbb{B}}\)
is an open subset of a Stein manifold, the groups
\(H^q(\mathbb{C}^n \setminus \overline{\mathbb{B}}, \Omega^p)\)
are either zero or infinite-dimensional.  
This completes the proof.
\end{proof}

\begin{lem}\label{V}
   For any \(p\) and for any \(1\le q\le n-2 \) one has
\begin{equation*}
H^q(V, \Omega^p) \cong H^q(\widehat{\mathbb{B}},\Omega^p).\end{equation*}
\end{lem}
\begin{proof}
Let $U$ be (biholomorphic to) a small ball containing $\overline{\sigma(\mathbb{B})}$. 
We compute the cohomology of $V$ using the Mayer--Vietoris sequence applied to 
the open cover $\{U, V\}$ of $\widehat{\mathbb{B}}$. Since
\(U \cap V\) is an annulus, Hartogs' extension theorem (see \cite[Chapter I, 3.28]{Demaillyag}) implies
\(\Omega^p(U \cap V) \cong \Omega^p(U)\). Hence, in the Mayer--Vietoris sequence associated with the open cover $\{U, V\}$, 
the map 
$$\Omega^p(U)\oplus\Omega^p(V)\rightarrow\Omega^p(U\cap V)$$
is surjective. Moreover, \(U\) is Stein, and therefore we have
\(H^q(U, \Omega^p) = 0\) for any \(q > 0\). 
By Lemma~\ref{annulus}, the Mayer--Vietoris sequence gives the desired isomorphisms:
$$0\rightarrow H^1(\widehat{\mathbb{B}},\Omega^p)\rightarrow  H^1(V,\Omega^p)\rightarrow 0,$$
$$H^{q-1}(U\cap V,\Omega^p)= 0\rightarrow H^q(\widehat{\mathbb{B}},\Omega^p)\rightarrow H^q(V,\Omega^p)\rightarrow 0= H^q(U\cap V,\Omega^p)$$
for any \(p\) and for any $2\le q\le n-2$ (in the case $n-2\ge 2$). 
\end{proof}

At this point, we can study the Mayer--Vietoris sequence coming from the open cover $\{U,V\}$ of $X$. From Lemma~\ref{annulus} and the fact that \(U\) is Stein, the 
Mayer--Vietoris sequence for \(X\) yields
\begin{equation}\label{cn2}
    \begin{split}
0\rightarrow\Omega^p(X)\rightarrow&\Omega^p(U)\oplus\Omega^p(V)\xrightarrow{\beta}\Omega^p(U_1)\oplus\Omega^p(U_2)\\
&\rightarrow H^1(X,\Omega^p)\rightarrow H^1(V,\Omega^p)\overset{\ref{V}}{\cong}H^1(\widehat{\mathbb{B}},\Omega^p)\rightarrow 0,
    \end{split}
\end{equation}
and for any $2\le q\le n-2$ (in the case $n-2\ge 2$):
\begin{equation}\label{cn3}
    \begin{split}
         H^{q-1}(U_1,\Omega^p)\oplus H^{q-1}(&U_2,\Omega^p)= 0 \rightarrow H^q(X,\Omega^p)\rightarrow H^q(V,\Omega^p)\\
    \rightarrow &H^q(U_1,\Omega^p)\oplus H^q(U_2,\Omega^p)= 0.
    \end{split}
\end{equation}
It follows from \eqref{cn3}, Lemma \ref{V}, and Dolbeault's theorem that for \(2\le q\le n-2\) (in the case $n-2\ge 2$),
\begin{equation}\label{cn1}
    \begin{split}
         H^{p,q}(X)&\cong H^q(X,\Omega^p)\\
    &\overset{(\ref{cn3})}{\cong}H^q(V,\Omega^p)\\
    &\overset{\ref{V}}{\cong}H^q(\widehat{\mathbb{B}},\Omega^p)\cong H^{p,q}(\widehat{\mathbb{B}}).
    \end{split}
\end{equation}

The following theorem follows from the results of Subsection~\ref{rewriting} and Section~\ref{isomorhism}.
\begin{thm}\label{betainfirstsection}
    The map $\beta$ in the Mayer--Vietoris sequence (\ref{cn2}) is an isomorphism for any $1\le p \le n$, whereas for $p=0$ its image has codimension one.
\end{thm}

We have thus proved the following theorem:
\begin{thm}\label{Hodge numbers}
    Let $X = X(\pi,\sigma)$ be a Kato manifold with Kato modification 
$\pi : \widehat{\mathbb{B}} \to \mathbb{B}$. 
Then we have the following relations for the Hodge numbers of $X$:
\begin{align*}
    h^{p,q}(X) &= h^{p,q}(\widehat{\mathbb{B}}),
\qquad q \neq 0,1,n-1,n\\
h^{p,0}(X) &= h^{n-p,n}(X) =
\begin{cases}
1, & p = 0 \\
0, & p \neq 0
\end{cases}\\
h^{p,1}(X) &= h^{n-p,n-1}(X) =
\begin{cases}
h^{0,1}(\widehat{\mathbb{B}}) + 1, & p = 0 \\
h^{p,1}(\widehat{\mathbb{B}}), & p \neq 0.
\end{cases}
\end{align*}
\end{thm}

\subsection{Simplifying the map \texorpdfstring{$\beta$}{beta} in the Mayer--Vietoris sequence}\label{rewriting}
In this subsection, we rewrite the map 
\begin{equation}\label{firstform}
    \begin{split}
        \beta : \Omega^p(U) \oplus \Omega^p(V)
&\longrightarrow \Omega^p(U_1) \oplus \Omega^p(U_2)\\
(\omega, \eta) &\longmapsto 
\big( \omega_{|_{U_1}} - \eta_{|_{U_1}}, \, 
       \omega_{|_{U_2}} - \eta_{|_{U_2}} \big)
    \end{split}
\end{equation}
appearing in (\ref{cn2}) in a more manageable form. Our goal is to describe its kernel and image, which we do by identifying them with those of a simpler map.  \\

First, we consider the map
\begin{equation}
    \begin{split}
        \beta_1 : \Omega^p(U) &\longrightarrow 
\Omega^p(U_1) \oplus \Omega^p(U_2)\\
\omega &\longmapsto 
\big( \omega_{|_{U_1}}, \, \omega_{|_{U_2}} \big).
    \end{split}
\end{equation}
Giving a holomorphic $p$-form 
$\omega \in \Omega^p(U)$ is equivalent to giving 
a holomorphic $p$-form $\tilde{\omega}$ on 
$\tilde{U} := U_1 \cup \overline{\mathcal{A}_I} \cup \sigma\pi(U_2)$, defined by
\[
\tilde{\omega} =
\begin{cases}
\omega_{|_{U_1}} & \text{on } U_1,\\[4pt]
(\pi^{-1}\sigma^{-1})^*\omega_{|_{\mathcal{A}_E}}\equiv\omega_{|_{\mathcal{A}_I}} & \text{on } \mathcal{A}_I,\\[4pt]
(\pi^{-1}\sigma^{-1})^*\omega_{|_{U_2}} & \text{on } \sigma\pi(U_2).
\end{cases}
\]
Hence, giving $\beta_1$ is equivalent to giving a new map
\begin{align*}
    \beta_1 : \Omega^p(\tilde{U}) 
&\longrightarrow 
\Omega^p(U_1) \oplus \Omega^p\big(\sigma\pi(U_2)\big)\\
\omega &\longmapsto 
\big( \omega_{|_{U_1}}, \, \omega_{|_{\sigma\pi(U_2)}} \big).
\end{align*}
Using Hartogs' extension theorem, we obtain the following identifications:
\begin{align*}
    \Omega^p(\tilde{U}) &\equiv  \Omega^p(\mathbb{B}_1), \\
    \Omega^p(U_1) &\equiv  \Omega^p(\mathbb{B}_1),\\
    \Omega^p\big(\sigma\pi(U_2)\big) &\equiv \Omega^p(\mathbb{B}_2).
\end{align*}
We can rewrite $\beta_1$ as
\begin{equation}
    \begin{split}
        \beta_1 : \Omega^p(\mathbb{B}_1)
&\longrightarrow 
\Omega^p(\mathbb{B}_1) \oplus \Omega^p(\mathbb{B}_2)\\
\omega &\longmapsto 
\big( \omega, \, \omega_{|_{\mathbb{B}_2}} \big).
    \end{split}
\end{equation}
Now we consider the map
\begin{equation}
    \begin{split}
        \beta_2 : \Omega^p(V) &\longrightarrow 
    \Omega^p(U_1) \oplus \Omega^p(U_2)\\
    \eta &\longmapsto 
    \big( \eta_{|_{U_1}}, \, \eta_{|_{U_2}} \big).
    \end{split}
\end{equation}
Again, we can identify this map with the following one:
\begin{align*}
    \beta_2 : \Omega^p(V) 
    &\longrightarrow 
    \Omega^p(U_1) \oplus \Omega^p\big(\sigma\pi(U_2)\big)\\
    \eta &\longmapsto 
    \big( \eta_{|_{U_1}}, \, (\pi^{-1}\sigma^{-1})^*\eta_{|_{U_2}} \big).
\end{align*}
From Hartogs’ extension theorem, we can further regard $\beta_2$ as the map
\begin{equation}
    \begin{split}
        \beta_2 : \Omega^p(\widehat{\mathbb{B}}\setminus\overline{\mathcal{A}_E})
    &\longrightarrow 
    \Omega^p(\mathbb{B}_1) \oplus \Omega^p(\mathbb{B}_2)\\
    \eta &\longmapsto 
    \big( \eta_{|_{\mathbb{B}_1}}, \, (\pi^{-1}\sigma^{-1})^*\eta \big).
    \end{split}
\end{equation}
Finally, we identify the map $\beta$ with 
\begin{equation}
    \begin{split}
        \beta : \Omega^p(\mathbb{B}_1)\oplus\Omega^p(\widehat{\mathbb{B}}\setminus\overline{\mathcal{A}_E})
    &\longrightarrow 
    \Omega^p(\mathbb{B}_1) \oplus \Omega^p(\mathbb{B}_2)\\
    (\omega,\eta) &\longmapsto 
    \bigg( \omega-\eta_{|_{\mathbb{B}_1}}, \,\omega_{|_{\mathbb{B}_2}}-(\pi^{-1}\sigma^{-1})^*\eta \bigg).
    \end{split}
\end{equation}
Asking \(\beta(\omega,\eta) = (a,b)\) 
is equivalent to requiring the following system:
\[
\begin{cases}
\omega = a + \eta_{|_{\mathbb{B}_1}},\\[6pt]
b - a_{|_{\mathbb{B}_2}} 
= \eta_{|_{\mathbb{B}_2}} 
- (\pi^{-1}\sigma^{-1})^*\eta.
\end{cases}
\]

For our purposes, we observe that studying the kernel and image of $\beta$ is 
equivalent to studying the kernel and image of the map $\tilde{\beta}$ defined by
\begin{align*}
    \tilde{\beta} : \Omega^p(\widehat{\mathbb{B}}\setminus\overline{\mathcal{A}_E})
    &\longrightarrow  \Omega^p(\mathbb{B}_2)\\
    \eta &\longmapsto 
    \eta_{|_{\mathbb{B}_2}}-(\pi^{-1}\sigma^{-1})^*\eta.
\end{align*} 
Indeed, in the end we will show that for $1\le p \le n$ the map $\tilde{\beta}$ is an isomorphism, which in turn implies that $\beta$ is also an isomorphism. If $p=0$, the image of $\tilde{\beta}$ is the space of holomorphic functions vanishing at $\sigma(0)$, and therefore the image of $\beta$ consists of pairs $(a,b)$ satisfying 
$a\big(\sigma(0)\big) = b\big(\sigma(0)\big)$. This subspace has codimension one in the codomain of $\beta$.

The endomorphism \(\tilde{\beta}\) is equivalently given by
\begin{equation}
    \begin{split}
        \tilde{\beta} : \Omega^p(\widehat{\mathbb{B}}\setminus\overline{\mathcal{A}_E})
    &\longrightarrow \Omega^p(\widehat{\mathbb{B}}\setminus\overline{\mathcal{A}_E})\\
    \eta &\longmapsto 
   (\sigma\pi)^*\iota^*\eta-\eta=(\iota\sigma\pi)^*\eta-\eta,
    \end{split}
\end{equation}
where $\iota$ is the inclusion $\iota:\mathbb{B}_2\hookrightarrow\widehat{\mathbb{B}}\setminus\overline{\mathcal{A}_E}$. 
By construction, we have isomorphisms
\begin{align*}
    \mathbf{1}_1 : \Omega^p(\widehat{\mathbb{B}}\setminus\overline{\mathcal{A}_E}) &\overset{\sim}{\longrightarrow}\Omega^p\big(\sigma^{-1}(\mathbb{B}_2)\big)\\
    \omega &\longmapsto (\pi^{-1})^*\omega,
\end{align*}
and 
\begin{align*}
    \mathbf{1}_2 : \Omega^p\big(\sigma^{-1}(\mathbb{B}_2)\big)
&\overset{\sim}{\longrightarrow}
\Omega^p(\widehat{\mathbb{B}}\setminus\overline{\mathcal{A}_E})\\
\omega &\longmapsto \pi^*\omega.
\end{align*}
Then, up to considering the composition 
$\mathbf{1}_1 \circ \tilde{\beta} \circ \mathbf{1}_2$, 
we can rewrite the map $\tilde{\beta}$ as
\begin{equation}
    \begin{split}
        \tilde{\beta} : \Omega^p\big(\sigma^{-1}(\mathbb{B}_2)\big)
    &\longrightarrow \Omega^p\big(\sigma^{-1}(\mathbb{B}_2)\big),\\
    \eta &\longmapsto (\pi\iota\sigma)^*\eta - \eta \equiv \gamma^*\eta - \eta,
    \end{split}
\end{equation}
and, using (\ref{preimageB2}), we can equivalently write
\begin{equation}
    \begin{split}
        \tilde{\beta} : \Omega^p (\mathbb{B}_{1-\epsilon})
    &\longrightarrow \Omega^p (\mathbb{B}_{1-\epsilon}),\\
    \eta &\longmapsto \gamma^*\eta - \eta.
    \end{split}
\end{equation}

\section{The map \texorpdfstring{$\mathrm{Id}-\gamma^*$}{Id-gamma*}}\label{isomorhism}
This section aims to prove the following theorem:
\begin{thm}\label{beta}
    Let $\gamma:\mathbb B\to\mathbb B$ be a holomorphic contraction at zero. Then, the map
    \begin{align*}
    \beta : \Omega^p(\mathbb{B}) &\to \Omega^p(\mathbb{B})\\
\eta &\mapsto \eta - \gamma^*\eta
\end{align*}
is an isomorphism for any $1\le p \le n$, whereas for $p=0$ its image has codimension one.
\end{thm}
\begin{rem}\label{mutandis}
    The statement of Theorem~\ref{beta} still holds if, in the definition of \(\beta\), we use any ball \(\mathbb{B}_r\) of radius \(r \in (0,1)\) instead of \(\mathbb{B}\). The proof applies \emph{mutatis mutandis}. Since the germ \(\gamma\) of a Kato manifold is a holomorphic contraction at the origin, Theorem~\ref{beta}, together with the results of Subsection~\ref{rewriting},
implies Theorem~\ref{betainfirstsection}.

\end{rem}

\begin{rem} 
By Subsection~\ref{rewriting}, the kernel of the operator
$\mathrm{Id}-\gamma^*$ is isomorphic to the kernel of the map in~\eqref{firstform},
and hence to $\Omega^p(X)$.
Following the approach of~\cite[Section~4]{MallFred}, we find that $\mathrm{Id}-\gamma^*$ is a Fredholm
operator of index zero.
The key point of the proof is the compactness of the map
$\gamma^*$.
In our setting, this follows directly from the property
$|\gamma(z)| \le |z|$ for any $z \in \mathbb{B}$, which implies that
$\gamma^*\left(\overline{\mathbb{B}_1\big(\Omega^p(\mathbb{B})\big)}\right)$
is bounded and contained in
$\overline{\mathbb{B}_1\big(\Omega^p(\mathbb{B})\big)}$.
The remainder of the argument then proceeds as in~\cite{MallFred}. 
As a consequence, if $p = 0$ the cokernel of $\mathrm{Id}-\gamma^*$ has dimension
one, since its index is zero, while if $1 \le p \le n$ it suffices to prove the
injectivity of $\mathrm{Id}-\gamma^*$ to conclude that it is an isomorphism.
However, in the following, we provide a more self-contained proof of
Theorem~\ref{beta}, avoiding the use of functional analysis.
\end{rem}

First, we study the relation 
\begin{equation}\label{gammainvariance}
\eta = \gamma^* \eta,
\end{equation}
where $\eta \in \Omega^p(\mathbb{B})$ is a holomorphic $p$-form with $1\le p \le n$, 
and $\gamma : \mathbb{B} \to \mathbb{B}$ is a holomorphic map.
We denote by \( I = (i_1, i_2, \ldots, i_p) \) an ordered subset of \( p \) distinct elements in \( \{1,\ldots,n\} \). 
If \( i_1 < i_2 < \ldots < i_p \), we say that \( I \) is \emph{ord}.  
We write $I\equiv J$ if $I$ and $J$ represent the same set. For any $I$, we have
\begin{align*}
    \gamma^*(dz_I)
&= d\gamma_{i_1} \wedge \ldots \wedge d\gamma_{i_p}\\
&= \sum_J 
\frac{\partial \gamma_{i_1}}{\partial z_{j_1}}
\cdot\ldots\cdot
\frac{\partial \gamma_{i_p}}{\partial z_{j_p}}
\, dz_J\\
&= \sum_J 
\frac{\partial \gamma_I}{\partial z_J} \, dz_J\\
&=\sum_{J\text{ord}} 
\sum_{L \equiv J}
\frac{\partial \gamma_I}{\partial z_L} 
(-1)^L
\, dz_J,
\end{align*}
where we used the notation
\[
\frac{\partial \gamma_I}{\partial z_J}
:= 
\frac{\partial \gamma_{i_1}}{\partial z_{j_1}}
\cdot\ldots\cdot
\frac{\partial \gamma_{i_p}}{\partial z_{j_p}}.
\]
If
\[
\eta = \sum_{I_1\text{ord}} \eta_{I_1} \, dz_{I_1},
\]
we obtain
\begin{align*}
    \gamma^* \eta
&= 
\sum_{I_1\text{ord}} 
\eta_{I_1}(\gamma)
\sum_{I\text{ord}} 
\sum_{L \equiv I}
\frac{\partial \gamma_{I_1}}{\partial z_L}
(-1)^L
\, dz_I\\
&=\sum_{I\text{ord}} \sum_{I_1\text{ord}} 
\eta_{I_1}(\gamma)
\sum_{L \equiv I}
\frac{\partial \gamma_{I_1}}{\partial z_L}
(-1)^L
\, dz_I.
\end{align*}
Therefore, for any $I$ the relation (\ref{gammainvariance}) yields
\begin{equation}
\eta_I
= 
\sum_{I_1\text{ord}}
\eta_{I_1}(\gamma)
\sum_{L_1 \equiv I}
\frac{\partial \gamma_{I_1}}{\partial z_{L_1}}
(-1)^{L_1}.
\end{equation}
Iterating this identity, for any $r>0$ we find:
\begin{align}\label{Injectivity}
    \eta_I &= \sum_{I_1\text{ord}} \Bigg(
  \sum_{I_2\text{ord}} 
  \eta_{I_2}(\gamma^2)
  \sum_{L_2 \equiv I_1} 
  \frac{\partial \gamma_{I_2}}{\partial z_{L_2}}(\gamma) (-1)^{L_2}
\Bigg)
\sum_{L_1 \equiv I} 
\frac{\partial \gamma_{I_1}}{\partial z_{L_1}} (-1)^{L_1}\notag\\
&=\sum_{I_1, I_2\text{ord}} 
\eta_{I_2}(\gamma^2)
\sum_{\substack{L_2 \equiv I_1 \\ L_1 \equiv I}}
\frac{\partial \gamma_{I_2}}{\partial z_{L_2}}(\gamma)
\frac{\partial \gamma_{I_1}}{\partial z_{L_1}}
(-1)^{L_1 + L_2}\notag\\
&\vdots\notag\\
&=\sum_{\substack{I_1, \ldots, I_r \\ \text{ord}}}
\sum_{\substack{L_r\equiv I_{r-1}\\ \ldots \\ L_1 \equiv I}}
\eta_{I_r}(\gamma^r)
\prod_{j=1}^r 
\left(\frac{\partial \gamma_{I_j}}{\partial z_{L_j}}\big(\gamma^{j-1}\big)\right)
(-1)^{L_j}.
\end{align}
Now we write explicitly $(\gamma^*)^r \eta$ for any $r > 0$. As before, we have 
\[
\gamma^*\eta
= 
\sum_{J_0, J_1\text{ord}}
\eta_{J_0} (\gamma)
\sum_{L_1 \equiv J_1}
\frac{\partial \gamma_{J_0}}{\partial z_{L_1}}
(-1)^{L_1} \, dz_{J_1}.
\]
Then,
\begin{align*}
    (\gamma^*)^2\eta
&=
\sum_{J_0, J_1 \text{ord}}
\eta_{J_0} ( \gamma^2)
\sum_{L_1 \equiv J_1}
\frac{\partial \gamma_{J_0}}{\partial z_{L_1}} ( \gamma)\
(-1)^{L_1}
\sum_{J_2 \text{ord}}
\sum_{L_2 \equiv J_2}
\frac{\partial \gamma_{J_1}}{\partial z_{L_2}}
(-1)^{L_2} \, dz_{J_2}\\
&=\sum_{\substack{J_0, J_1, J_2\\ \text{ord}}}
\sum_{\substack{L_1 \equiv J_1 \\ L_2 \equiv J_2}}
\eta_{J_0} ( \gamma^2)\cdot
\frac{\partial \gamma_{J_0}}{\partial z_{L_1}}( \gamma)\cdot
\frac{\partial \gamma_{J_1}}{\partial z_{L_2}}
(-1)^{L_1 + L_2} \, dz_{J_2}.
\end{align*}
Iterating, for any $r>0$ we obtain
\begin{equation}\label{Powers}
(\gamma^*)^r \eta
=
\sum_{\substack{J_0, \ldots, J_r \\ \text{ord}}}
\sum_{\substack{L_1 \equiv J_1 \\ \ldots \\ L_r \equiv J_r}}
\eta_{J_0} ( \gamma^r)
\prod_{l=0}^{r-1}
\left(
\frac{\partial \gamma_{J_l}}{\partial z_{L_{l+1}}}
\big( \gamma^{r-1-l}\big)
\right)
(-1)^{L_{l+1}} \, dz_{J_r}.
\end{equation}

The next Lemma~\ref{control} is a simple observation, yet it plays a crucial role in the proof of Proposition~\ref{small contraction}. This proposition forms the core of the proof of Theorem~\ref{beta} for $1\le p \le n$.

\begin{lem}\label{control}
    Let $\gamma:\mathbb{B}\to\mathbb{B}$ be a holomorphic contraction at zero such that 
\[
|\gamma(z)| \le C |z| \quad \forall \,z \in \mathbb{B}.
\]
Then for any $1\le i,j\le n$ we have
\[
\left|\frac{\partial \gamma_i}{\partial z_j}(0)\right| \le C.
\]
\end{lem}
\begin{proof}
By definition, we have
\[
d_0\gamma \cdot v = \lim_{t \to 0} \frac{\gamma(tv)}{t} \quad \forall\ v \in \mathbb{C}^n.
\]
Hence,
\[
|d_0\gamma \cdot v| \le C\ |v| \quad \forall\ v \in \mathbb{C}^n,
\]
which allows us to conclude.
\end{proof}

\begin{prop}\label{small contraction}
Let $\gamma:\mathbb B\to\mathbb B$ be a holomorphic contraction at $0$ such that
\[
|\gamma(z)|\le C|z|\qquad\forall\, z\in\mathbb B.
\]
If $1\le p \le n$ and the constant $C>0$ satisfies
\begin{equation*}
    (2C)^p\frac{n!}{(n-p)!}<1,
\end{equation*}
then the map
\begin{align*}
    \beta : \Omega^p(\mathbb{B}) &\to \Omega^p(\mathbb{B})\\
\eta &\mapsto \eta - \gamma^*\eta
\end{align*}
is an isomorphism.
\end{prop}

\begin{proof}
    Let 
\[
\eta = \sum_{I\text{ord}} \eta_I\, dz_I \in \Omega^p(\mathbb{B})
\]
be a holomorphic $p$-form on $\mathbb{B}$ such that $\eta = \gamma^* \eta$. 
As in (\ref{Injectivity}), for any $r>0$, we have
\[
\eta_I = 
\sum_{\substack{I_1, \ldots, I_r \\ \text{ord}}}
\sum_{\substack{L_r \equiv I_{r-1} \\ \ldots \\ L_1 \equiv I}}
\eta_{I_r}(\gamma^r)
\prod_{j=1}^r 
\left(\frac{\partial \gamma_{I_j}}{\partial z_{L_j}}\big(\gamma^{j-1}\big)\right)
(-1)^{L_j}.
\]
By hypotesis and Lemma \ref{control},  there exists a small positive number $R \ll \tfrac{1}{2}$ such that for any $1\le i,j\le n $,
\[
\left| \frac{\partial \gamma_i}{\partial z_j}(z) \right|
\le 2C
\qquad \forall\ z \in \mathbb{B}_{R}.
\]
Set
\[
M := \max_I \|\eta_I\|_{L^\infty(\overline{\mathbb{B}}_{1/2})}.
\]
For any $I,L$, we have
\[
\left| \frac{\partial \gamma_I}{\partial z_L} \right| 
\le (2C)^p \quad \text{on } \mathbb{B}_R.
\]
For any $r > 0$ and $p$-multi–index $I$,
\begin{align}
    |\eta_I| &\le M (2C)^{p r}\sum_{\substack{I_1, \ldots, I_r \\ \text{ord}}}
\sum_{\substack{L_r \equiv I_{r-1} \\ \ldots \\ L_1 \equiv I}}1\notag\\
&=M (2C)^{p r} \binom{n}{p}^r (p!)^r\notag\\
&= M \left((2C)^p \frac{n!}{(n - p)!} \right)^r\quad \text{on}\ \ \mathbb{B}_R.\label{RS}
\end{align}
By hypothesis the right-hand side of (\ref{RS}) tends to zero as $r$ goes to $+\infty$, and hence each $\eta_I$ vanishes on $\mathbb{B}_R$.  
Consequently, $\eta \equiv 0$. The injectivity is proved.\\
Let $\eta \in \Omega^p(\mathbb{B})$ be a holomorphic $p$-form and $M$ as before.
We will prove that the series
\begin{equation}\label{series}
S_\eta := \sum_{m=0}^\infty (\gamma^*)^m \eta
\end{equation}
converges uniformly on compact subsets of $\mathbb{B}$.
This shows that $S_\eta$ is a holomorphic $p$-form on $\mathbb{B}$ and then
$\beta(S_\eta)=\eta$.\\
Let $R_0\in(0,1)$ be fixed. We prove the uniform convergence of the
series (\ref{series}) on $\overline{\mathbb{B}}_{R_0} \subset \mathbb{B}$.
Set
\[
C_{R_0} := \max_{i,j}
\left\|
\frac{\partial \gamma_i}{\partial z_j}
\right\|_{L^\infty(\overline{\mathbb{B}}_{R_0})},
\]
and choose $r_0\equiv r_0(R_0)$ such that for any $r\ge r_0$,
\[
\gamma^r(\overline{\mathbb{B}}_{R_0}) \subset \mathbb{B}_R .
\]
For any $r>0$ and $p$-multi-index $J$, on $\overline{\mathbb{B}}_{R_0}$ we have
\begin{align}\label{Mtest}
    \Big|\left((\gamma^*)^{\,r+r_0}\eta\right)_{J}\Big|&= \left| \sum_{\substack{J_0, \ldots, J_{r+r_0-1}\\ \text{ord}}}
\sum_{\substack{L_1 \equiv J_1 \\ \ldots \\ L_{r+r_0} \equiv J}}
\eta_{J_0} ( \gamma^{\,r+r_0})
\prod_{l=0}^{r+r_0-1}
\left(
\frac{\partial \gamma_{J_l}}{\partial z_{L_{l+1}}}
\big( \gamma^{r+r_0-1-l}\big)
\right)(-1)^{L_{l+1}}
\right| \notag\\
&\le M \sum_{\substack{J_0,\ldots,J_{r+r_0-1}\\ \mathrm{ord}}}
\sum_{\substack{L_1\equiv J_1 \\ \cdots \\ L_{r+r_0}\equiv J}}
\left|\prod_{l=0}^{r-1}
\left(
\frac{\partial \gamma_{J_l}}{\partial z_{L_{l+1}}}
\big(\gamma^{r+r_0-1-l}\big)
\right)
\prod_{l=r}^{r+r_0-1}
\left(
\frac{\partial \gamma_{J_l}}{\partial z_{L_{l+1}}}
\big(\gamma^{r+r_0-1-l}\big)
\right)
\right|\notag\\
&\le M (2C)^{pr}\, C_{ R_0}^{\,p r_0}
\left(\frac{n!}{(n-p)!}\right)^{r+r_0}\notag\\
&=M \left( C_{R_0}^p \frac{n!}{(n-p)!}\right)^{r_0}\left((2C)^p \frac{n!}{(n - p)!} \right)^r.
\end{align}
Since 
$$\left((2C)^p \frac{n!}{(n - p)!} \right) < 1,$$
the geometric series given by the terms on the right-hand side of (\ref{Mtest})
is convergent. By the Weierstrass M-test, $S_\eta$ converges uniformly on 
$\overline{\mathbb{B}}_{R_0}$. From the arbitrariness of $R_0 \in (0,1)$, we conclude that the series converges uniformly on compact subsets of $\mathbb{B}$, and therefore (\cite[Corollary 2.2.4]{Hormander1990}) defines a holomorphic $p$-form $S_\eta \in \Omega^p(\mathbb{B})$.

\end{proof}

We are now in a position to prove Theorem~\ref{beta} for $1\le p\le n$.

\begin{thm}
    If $\gamma:\mathbb B\to\mathbb B$ is a holomorphic contraction at $0$, 
then for any $1\le p \le n$ the map
\begin{align*}
    \beta : \Omega^p(\mathbb{B}) &\to \Omega^p(\mathbb{B})\\
    \eta &\mapsto \eta - \gamma^* \eta
\end{align*}
is an isomorphism.
\end{thm}
\begin{proof}
From Lemma~\ref{control} and Proposition~\ref{small contraction}, it follows that if $d \gg 1$ and $1\le p \le n$, the endomorphism 
$$\mathrm{Id} - (\gamma^*)^d:\Omega^p(\mathbb{B}) \to \Omega^p(\mathbb{B})$$
is an isomorphism.
By the decomposition
\begin{align*}
    \mathrm{Id} - (\gamma^*)^d 
&= (\mathrm{Id} - \gamma^*)
\bigl(\mathrm{Id} + \gamma^* + \ldots + (\gamma^*)^{d-1}\bigr)\\
&= \bigl(\mathrm{Id} + \gamma^* + \ldots + (\gamma^*)^{d-1}\bigr)
(\mathrm{Id} - \gamma^*),
\end{align*}
we deduce that $\mathrm{Id} - \gamma^*$ is also an isomorphism.
\end{proof}

The following Proposition~\ref{p=0} concludes the proof of Theorem~\ref{beta}.

\begin{prop}\label{p=0}
    If $\gamma:\mathbb B\to\mathbb B$ is a holomorphic contraction at $0$, then the image of the map
\begin{align*}
    \beta : \mathcal{O}(\mathbb{B}) &\to \mathcal{O}(\mathbb{B})\\
f &\mapsto f - \gamma^*f
\end{align*}
is given by the holomorphic functions on $\mathbb{B}$ vanishing at zero.
\end{prop}
\begin{proof}
If $f \in \mathcal{O}(\mathbb{B})$, then $\beta(f)(0) = f(0) - f(\gamma(0)) = 0$, so one inclusion is clear.\\
Let $g \in \mathcal{O}(\mathbb{B})$ be a holomorphic function on $\mathbb{B}$ such that $g(0)=0$.
We will prove the uniform convergence of the series
\begin{equation}
f_g := \sum_{m=0}^\infty (\gamma^*)^m g
\end{equation}
on compact subsets of $\mathbb{B}$.
This will show that $f_g$ is a holomorphic function on $\mathbb{B}$, and consequently $\beta(f_g) = g$.\\
Let $R\in(0,1)$ be fixed. We prove the uniform convergence of the
series on $\overline{\mathbb{B}}_R \subset \mathbb{B}$.\\
Since $g(0)=0 $, there exist holomorphic functions $g_1, g_2, \ldots, g_n \in \mathcal{O}(\mathbb{B})$ such that
\[
g(z) = \sum_{i=1}^n z_i g_i(z),
\]  
see, for instance, \cite[Theorem 1.1.17]{Huybrechts}.
Set
\[
M\equiv M(R) := \max_{1\le i\le n} \| g_i \|_{L^\infty(\overline{\mathbb{B}}_R)}.
\]  
Since $\gamma$ is a contraction, there exists $C \in (0,1)$ such that
\[
|\gamma(z)| \le C |z| \quad \forall\,z \in \mathbb{B}.
\]
Then, for any $z \in \overline{\mathbb{B}}_R$,
\begin{align}\label{Mtest2}
\left|(\gamma^*)^m g(z)\right| 
&= |g\big(\gamma^m(z)\big)| \notag\\
&\le \sum_{i=1}^n \big|\big(\gamma^m(z)\big)_i\ g_i\big(\gamma^m(z)\big)\big| \notag\\
&\le M n |\gamma^m(z)| \notag\\
&\le M n C^m.
\end{align}  
Since $C < 1$, the geometric series given by the terms on the right-hand side of (\ref{Mtest2}) is convergent. By the Weierstrass M-test, the series $\sum_{m=0}^{+\infty} (\gamma^*)^m g$ converges uniformly on 
$\overline{\mathbb{B}}_R$. From the arbitrariness of $R \in (0,1)$, we conclude that $f_g$ converges uniformly on compact subsets of $\mathbb{B}$, and therefore (\cite[Corollary 2.2.4]{Hormander1990}) defines a holomorphic function $f_g \in \mathcal{O}(\mathbb{B})$.
\end{proof}

\section{Hodge decomposition}\label{finalsection}
In this section, we prove the Hodge decomposition for Kato manifolds. To this end, we compare the cohomology of a Kato manifold with that of the space 
\(\widehat{\mathbb{CP}^n}\) induced by the Kato modification, on which the Hodge decomposition is known to hold.\\

Let $X(\pi,\sigma)$ be a Kato manifold with proper modification 
$\pi : \widehat{\mathbb B} \to \mathbb B$ and let $\widehat H$ denote the modification of the primary Hopf manifold $H$ obtained from $X$ via small deformation \cite[Theorem 1]{Kato77}. As noted in \cite{IOPR22}, any proper modification 
$\pi : \widehat{\mathbb B} \to \mathbb B$ naturally induces proper modifications 
$\pi : \widehat{\mathbb{C}^n} \to \mathbb C^n$ and 
$\pi : \widehat{\mathbb{CP}^n} \to \mathbb{CP}^n$, 
where $\widehat{\mathbb{CP}^n}$ is the compactification of 
$\widehat{\mathbb C^n}$ by adding a hyperplane at infinity. 
In particular, we have natural holomorphic embeddings that give the following commutative diagram:
\begin{equation}
\begin{tikzcd}
\widehat{\mathbb B} \arrow[r, hook ] \arrow[d, "\pi"'] & \widehat{\mathbb{CP}^n} \arrow[d, "\pi"] \\
\mathbb B \arrow[r, hook] & \mathbb{CP}^n.
\end{tikzcd}
\end{equation}

In the next proposition, we relate certain Hodge and Betti numbers of $\widehat{\mathbb{CP}^n}$ to those of $\widehat{\mathbb{B}}$.

\begin{prop}\label{noWFT}
For any \(p\) and any \(1\le q \le n-2 \), we have
\[
h^{p,q}(\widehat{\mathbb{CP}^n}) 
    =  h^{p,q}(\mathbb{CP}^n)+h^{p,q}(\widehat{\mathbb{B}}).
\]
Moreover, for any \(1\le k \le 2n-2 \), we have 
\[
b_k(\widehat{\mathbb{CP}^n}) 
    = b_k(\mathbb{CP}^n)+b_k(\widehat{\mathbb{B}}).
\]
\end{prop}

\begin{proof}
We relate the Dolbeault cohomology of $\widehat{\mathbb{CP}^n}$ to that of 
$\widehat{\mathbb{B}}$ by applying the Mayer--Vietoris sequence in sheaf cohomology 
to the open cover
\[
\bigl(\widehat{\mathbb{CP}^n} \setminus \overline{\widehat{\mathbb{B}}},\,
      \widehat{\mathbb{B}}_2\bigr)
\]
of $\widehat{\mathbb{CP}^n}$, where $\widehat{\mathbb{B}}_2 := \pi^{-1}(\mathbb{B}_2)$.
By construction, we have the following identifications:
\begin{align*}
    \widehat{\mathbb{CP}^n} \setminus \overline{\widehat{\mathbb{B}}}
    &\equiv \mathbb{CP}^n \setminus \overline{\mathbb{B}};\\
    \bigl(\widehat{\mathbb{CP}^n} \setminus \overline{\widehat{\mathbb{B}}}\bigr)
    \cap \widehat{\mathbb{B}}_2
    &\equiv \mathcal{A},
\end{align*}
where $\mathcal{A} \equiv \mathbb{B}_2 \setminus \overline{\mathbb{B}}$ is an annulus.\\
In what follows, we will use the vanishing of the sheaf cohomology of \(\mathcal{A}\) proved in Lemma~\ref{annulus}, as well as the Dolbeault theorem, without explicitly recalling them each time.
If $2\le q \le n-2 $, the Mayer--Vietoris sequence yields
\begin{align*} 
H^{q-1}(\mathcal{A},\Omega^p)= 0 \longrightarrow H^q(\widehat{\mathbb{CP}^n},\Omega^p) \longrightarrow &H^q(\mathbb{CP}^n \setminus \overline{\mathbb{B}},\Omega^p) \oplus H^q(\widehat{\mathbb{B}}_2,\Omega^p)\\
\longrightarrow H^{q}(\mathcal{A},\Omega^p)= 0. \end{align*}
Applying the same sequence to the open cover 
\((\mathbb{CP}^n \setminus \overline{\mathbb{B}},\, \mathbb{B}_2)\) of 
$\mathbb{CP}^n$ we find
\[
H^q(\mathbb{CP}^n \setminus \overline{\mathbb{B}},\Omega^p)
    \cong H^q(\mathbb{CP}^n,\Omega^p).
\]
Thus,
\[
h^{p,q}(\widehat{\mathbb{CP}^n})
    =  h^{p,q}(\mathbb{CP}^n)+h^{p,q}(\widehat{\mathbb{B}}_2).
\]
We now consider the first part of the Mayer--Vietoris sequence.  
After applying Hartogs' extension theorem, we obtain
\begin{align*}
0 \longrightarrow\,
&\Omega^p(\widehat{\mathbb{CP}^n})
\longrightarrow
\Omega^p(\mathbb{CP}^n)
\oplus
\Omega^p(\widehat{\mathbb{B}}_2)
\xrightarrow{\ \beta\ }
\Omega^p(\mathbb{B}_2)
\\
&\longrightarrow
H^1(\widehat{\mathbb{CP}^n},\Omega^p)
\longrightarrow
H^1(\mathbb{CP}^n \setminus \overline{\mathbb{B}},\Omega^p)
\oplus
H^1(\widehat{\mathbb{B}}_2,\Omega^p)
\longrightarrow 0 .
\end{align*}
As before,
\[
H^1(\mathbb{CP}^n \setminus \overline{\mathbb{B}},\Omega^p)
    \cong H^1(\mathbb{CP}^n,\Omega^p).
\]
Moreover, the map $\beta$ is always surjective and for $p\neq 0$ is an isomorphism.  
Therefore, the sequence above shows that
\[
h^{p,1}(\widehat{\mathbb{CP}^n})
    = h^{p,1}(\mathbb{CP}^n)+h^{p,1}(\widehat{\mathbb{B}}_2),
\]
which concludes the proof for the Hodge numbers.

Since an annulus $\mathcal{A}$ retracts onto the sphere
$\mathbb{S}^{2n-1}$, its de Rham cohomology vanishes in all degrees
$1\le k\le 2n-2 $. Applying the Mayer--Vietoris sequence for de Rham cohomology to the same open covers used above, both for $\widehat{\mathbb{CP}^n}$ and for
$\mathbb{CP}^n$, we obtain, for any $2\le k\le 2n-2 $,
\begin{align*}
    0 = H^{k-1}(\mathcal{A})
   \longrightarrow H^{k}(\widehat{\mathbb{CP}^n})
   \longrightarrow
     H^{k}\!(\mathbb{CP}^n\setminus
                   \overline{\mathbb{B}})
     \oplus
     H^{k}(\widehat{\mathbb{B}}_2)
   \longrightarrow H^{k}(\mathcal{A}) = 0
\end{align*}
and
\[
0 \longrightarrow H^{k}(\mathbb{CP}^n)
   \longrightarrow H^{k}(\mathbb{CP}^n\setminus\overline{\mathbb{B}})
   \longrightarrow 0.
\]
We then conclude that, for \(2\le k\le 2n-2 \),
\[
b_k(\widehat{\mathbb{CP}^n})
    = b_k(\mathbb{CP}^n)+b_k(\widehat{\mathbb{B}}_2) .
\]
For $k = 1$, the first part of the Mayer--Vietoris sequence gives
\begin{align*}
    0 \longrightarrow H^{0}(\widehat{\mathbb{CP}^n}) \longrightarrow 
&H^{0}(\mathbb{CP}^n\setminus\overline{\mathbb{B}})\oplus H^{0}(\widehat{\mathbb{B}}_2) \xrightarrow{\beta} 
H^{0}(\mathcal{A}) \\
&\longrightarrow H^{1}(\widehat{\mathbb{CP}^n}) \longrightarrow
H^{1}(\mathbb{CP}^n\setminus\overline{\mathbb{B}})\oplus H^{1}(\widehat{\mathbb{B}}_2) \longrightarrow 0,
\end{align*}
where $\beta$ is surjective and, as before, $H^{1}(\mathbb{CP}^n\setminus\overline{\mathbb{B}})\cong H^{1}(\mathbb{CP}^n)$.
Then we find 
\[
b_1(\widehat{\mathbb{CP}^n})
    = b_1(\mathbb{CP}^n)+b_1(\widehat{\mathbb{B}}_2) .
\]

\end{proof}

The arguments used in Proposition~\ref{noWFT} are more general and do not depend on the specific structure of $\widehat{\mathbb{CP}^n}$; this leads to the following result:

\begin{prop}\label{generalization}
Let $X$ be a compact complex manifold, and let 
$\pi : \widehat{X} \to X$ be a proper holomorphic modification of $X$ at a point 
$p \in X$.  
Let $\mathbb{B}$ be a ball in $X$ centered at $p$, and let 
$\pi : \widehat{\mathbb{B}} \to \mathbb{B}$ denote the induced modification on \(\mathbb{B}\).  
Then, for any $p$ and any $1\le q\le n-2 $, we have
\[
h^{p,q}(\widehat{X})
    =h^{p,q}(X)+ h^{p,q}(\widehat{\mathbb{B}}).
\]
Moreover, for any $1\le k\le 2n-2 $, we have
\[b_k(\widehat{X})=b_k(X)+b_k(\widehat{\mathbb{B}}).\]
\end{prop}
As a direct application of Proposition~\ref{generalization}, we obtain the following corollary:

\begin{cor}\label{HodgeBetti}
    Let $X$ be a compact complex manifold, and let 
$\pi : \widehat{X} \to X$ be a proper holomorphic modification of $X$ at a point 
$p \in X$.  
Let $\widehat{\mathbb{CP}^n}$ be the modification at a point of $\mathbb{CP}^n$ induced by $\pi$.
Then, for any $p$ and any $q$, we have
\[
h^{p,q}(\widehat{X})
    = h^{p,q}(X)
      + h^{p,q}(\widehat{\mathbb{CP}^n})-h^{p,q}(\mathbb{CP}^n).
\]
Moreover, for any $k$, we have
\[b_k(\widehat{X})=b_k(X)
      + b_k(\widehat{\mathbb{CP}^n})-b_k(\mathbb{CP}^n).\]
\end{cor}
\begin{proof}
We start with the Hodge numbers. 
For \(1 \leq q \leq n-2\), the statement follows directly from 
Proposition~\ref{generalization}, applied twice: once to \(\widehat{X}\) and once to 
\(\widehat{\mathbb{CP}^n}\).
For \(q = n-1\), the result follows from Serre duality and the case \(q=1\):
\begin{align*}
h^{p,n-1}(\widehat{X})
    &= h^{n-p,1}(\widehat{X}) \\
    &=  h^{n-p,1}(X)
        + h^{n-p,1}(\widehat{\mathbb{CP}^n})
        - h^{n-p,1}(\mathbb{CP}^n) \\
    &= h^{p,n-1}(X)
        + h^{p,n-1}(\widehat{\mathbb{CP}^n})
        - h^{p,n-1}(\mathbb{CP}^n).
\end{align*}
For \(q = 0\), the claim follows immediately from Corollary~\ref{p00q}, 
and applying Serre duality yields the case \(q = n\).\\
Using analogous arguments, we obtain the corresponding relation for the Betti numbers.

\end{proof}

The next step is to compare the Hodge and Betti numbers of $X$ to those of $\widehat{\mathbb{CP}^n}$.

\begin{thm}\label{Hodge numbers21}
   Let $X = X(\pi,\sigma)$ be a Kato manifold with Kato modification 
$\pi : \widehat{\mathbb{B}} \to \mathbb{B}$. 
Then the Hodge numbers of $X$ are given by:
\[
h^{p,q}(X)=
\begin{cases}
1, & (p,q)\in\{(0,0),(0,1),(n,n-1),(n,n)\}\\
h^{p,q}(\widehat{\mathbb{CP}^n}) - h^{p,q}(\mathbb{CP}^n), &\text{otherwise}.
\end{cases}
\]
Moreover, the Betti numbers of $X$ satisfy:
\[
b_k(X)=
\begin{cases}
1, & k\in \{0,1,2n-1,2n\}\\
b_k(\widehat{\mathbb{CP}^n}) - b_k(\mathbb{CP}^n),& \text{otherwise}.
\end{cases}
\]
\end{thm}

\begin{proof}
We start by comparing the Hodge numbers.  
Using Corollary~\ref{p00q}, Theorem~\ref{Hodge numbers}, Proposition~\ref{noWFT}, and Serre duality, we obtain the desired relations between the  
Hodge numbers of \(X\) and those of \(\widehat{\mathbb{CP}^n}\):
\begin{align*}
    h^{1,n-1}(X)&\overset{\ref{Hodge numbers}}{=} h^{n-1,1}(\widehat{\mathbb{B}})\\
    &\overset{\ref{noWFT}}{=}h^{n-1,1}(\widehat{\mathbb{CP}^n})
    = h^{1,n-1}(\widehat{\mathbb{CP}^n});\\
      h^{0,n-1}(X)&\overset{\ref{Hodge numbers}}{=} h^{n,1}(\widehat{\mathbb{B}})\\
    &\overset{\ref{noWFT}}{=} h^{n,1}(\widehat{\mathbb{CP}^n})
    = h^{0,n-1}(\widehat{\mathbb{CP}^n});\\
    h^{0,n}(X) &\overset{\ref{Hodge numbers}}{=} 0 \overset{\ref{p00q}}{=} h^{0,n}(\widehat{\mathbb{CP}^n});\\
    h^{k,0}(X) &\overset{\ref{Hodge numbers}}{=} 0 \overset{\ref{p00q}}{=} h^{k,0}(\widehat{\mathbb{CP}^n})
    \qquad \text{for } 1\le k\le n ;\\
    h^{0,1}(X)
    &\overset{\ref{Hodge numbers}}{=} 1 + h^{0,1}(\widehat{\mathbb{B}})\\
    &\overset{\ref{noWFT}}{=} 1 + h^{0,1}(\widehat{\mathbb{CP}^n})
    \overset{\ref{p00q}}{=} 1;
\end{align*}
and for any $2\le p+q\le n$ with $q\ne 0,n-1,n$,
\begin{align*}
    h^{p,q}(X)
    &\overset{\ref{Hodge numbers}}{=} h^{p,q}(\widehat{\mathbb{B}})\\
    &\overset{\ref{noWFT}}{=} h^{p,q}(\widehat{\mathbb{CP}^n})
      - h^{p,q}(\mathbb{CP}^n).
\end{align*}
Let \(\widehat{H}\) be the modification of a primary Hopf manifold \(H\) obtained via small deformation from \(X\). Since \(\widehat{H}\) and \(X\) are diffeomorphic, they have the same Betti numbers. Recall also (see, $\textsl{e.g.}$, \cite{istrati2025propertieshopfmanifolds}) that the Betti numbers \(b_k\) of a primary Hopf manifold are equal to one for \(k = 0,1,2n-1,2n\) and zero otherwise.  
For \(2\le k \le 2n-2 \), Corollary~\ref{HodgeBetti} gives
\begin{align*}
    b_k(X) &= b_k(\widehat{H}) \\
    &\overset{\ref{HodgeBetti}}{=}  b_k(H)+ b_k(\widehat{\mathbb{CP}^n}) -b_{k}(\mathbb{CP}^n)=b_k(\widehat{\mathbb{CP}^n}) -b_{k}(\mathbb{CP}^n).
\end{align*}
Finally, we note that the first Betti number of \(X\) is equal to one, as its universal cover is a \(\mathbb{Z}\)-cover.
\end{proof}

\begin{rem}\label{corollario}
We point out that Theorem~\ref{Hodge numbers21} can also be deduced, instead of using Proposition~\ref{generalization}, by applying the Weak Factorization Theorem \cite[Theorem~0.3.1]{AKMWo02} together with the blow-up formulas in Subsection~\ref{blowformulas}.  
Indeed, by the Weak Factorization Theorem, the Kato modification 
\(\pi\) decomposes as a sequence of blow-ups and blow-downs along irreducible
smooth centers. Using this fact and Theorem~\ref{blowupD}, we can relate the 
Dolbeault cohomology of \(\widehat{\mathbb{B}}\) to that of 
\(\widehat{\mathbb{CP}^n}\).
In particular, for any \(p\) and any \(q \neq 0\), we have
\begin{equation}
     h^{p,q}(\widehat{\mathbb{B}})
    = h^{p,q}(\widehat{\mathbb{CP}^n}) - h^{p,q}(\mathbb{CP}^n).
\end{equation}
Indeed, \(\pi\) is a composition of blow-ups and blow-downs, so each  \(h^{p,q}(\widehat{\mathbb{B}})\) can be expressed as a sum involving 
only certain Hodge numbers of the centers along which these blow-ups 
and blow-downs are performed, together with the \((p,q)\)-Hodge number
of the ball \(\mathbb{B}\).
More precisely, if \(\{Z_i\}_{i \in I}\) are the centers of the blow-ups 
and \(\{W_j\}_{j \in J}\) those of the blow-downs, then, applying 
Theorem~\ref{blowupD} iteratively, we obtain
\begin{align*}
    h^{p,q}(\widehat{\mathbb{B}}) 
&= \sum_{i \in I} \sum_{k=1}^{\mathrm{codim}(Z_i)-1} h^{p-k,q-k}(Z_i) 
- \sum_{j \in J} \sum_{k=1}^{\mathrm{codim}(W_j)-1} h^{p-k,q-k}(W_j) 
+ h^{p,q}(\mathbb{B})\\
&= \sum_{i \in I} \sum_{k=1}^{\mathrm{codim}(Z_i)-1} h^{p-k,q-k}(Z_i) 
- \sum_{j \in J} \sum_{k=1}^{\mathrm{codim}(W_j)-1} h^{p-k,q-k}(W_j) 
+ h^{p,q}(\mathbb{CP}^n)- h^{p,q}(\mathbb{CP}^n)\\
&= h^{p,q}(\widehat{\mathbb{CP}^n}) - h^{p,q}(\mathbb{CP}^n),
\end{align*}
since the Dolbeault cohomology of the ball vanishes for \(q \neq 0\).\\
Applying the same argument but using Theorem~\ref{blowupd}, we obtain the analogous relation for Betti numbers. That is, for any \(k \neq 0\),
\begin{equation}
    \begin{split}
         b_{k}(\widehat{\mathbb{B}})
    = b_{k}(\widehat{\mathbb{CP}^n}) - b_{k}(\mathbb{CP}^n).
    \end{split}
\end{equation}
\end{rem}

\begin{figure}[!t]
    \centering
    \caption{ Hodge diamond of \(X = X(\pi, \sigma)\), where \(h^{p,q} := h^{p,q}(\widehat{\mathbb{B}})\).}
  \vspace{3mm} 
  
\begin{tikzpicture}[any node/.style={font=\scriptsize, inner sep=1pt}]

\def\dx{0.7}   
\def\dy{0.8}   

\node (p0q0) at (0,0) {$1$};

\node (p1q0) at (-1*\dx,-1*\dy) {0};
\node (p0q1) at (1*\dx,-1*\dy) {1};

\node (p2q0) at (-2*\dx,-2*\dy) {0};
\node (p1q1) at (0,-2*\dy) {$h^{1,1}$};
\node (p0q2) at (2*\dx,-2*\dy) {$h^{0,2}$};

\node (A) at (-2.8*\dx,-2.8*\dy) {.};
\node (B) at (-3*\dx,-3*\dy) {.};
\node (C) at (-3.2*\dx,-3.2*\dy) {.};
\node (D) at (2.8*\dx,-2.8*\dy) {.};
\node (E) at (3*\dx,-3*\dy) {.};
\node (F) at (3.2*\dx,-3.2*\dy) {.};
\node (G) at (-0.8*\dx,-2.8*\dy) {.};
\node (H) at (-1*\dx,-3*\dy) {.};
\node (I) at (-1.2*\dx,-3.2*\dy) {.};
\node (L) at (0.8*\dx,-2.8*\dy) {.};
\node (M) at (1*\dx,-3*\dy) {.};
\node (N) at (1.2*\dx,-3.2*\dy) {.};

\node (pn1q0) at (-4*\dx,-4*\dy) {0};
\node (pn2q1) at (-2*\dx,-4*\dy) {$h^{n-2,1}$};
\node (dotsRightgiu') at (0*\dx,-4.15*\dy) {$.\,\,.\,\,.$};

\node (p1qn2) at (2*\dx,-4*\dy) {$h^{1,n-2}$};
\node (p0qn1) at (4*\dx,-4*\dy) {$h^{0,n-1}$};

\node (pnq0) at (-5*\dx,-5*\dy) {0};
\node (pn1q1) at (-3*\dx,-5*\dy) {$h^{n-1,1}$};

\node (dotsn2) at (0,-5.15*\dy) {$.\,\,.\,\,.\,\,.\,\,.$};

\node (p1qn1) at (3*\dx,-5*\dy) {$h^{1,n-1}$};
\node (p0qn)  at (5*\dx,-5*\dy) {$h^{0,n}$};

\end{tikzpicture}

\end{figure}

We can also relate the Hodge and Betti numbers of $X$ to those of $\widehat{H}$.
\begin{thm}\label{samenumbers}
     Let $X = X(\pi,\sigma)$ be a Kato manifold with Kato modification
   $\pi : \widehat{\mathbb{B}} \to \mathbb{B}$. Let $\widehat{H}$ be the modification of a primary Hopf manifold $H$ obtained via small deformation from $X$. Then $\widehat{H}$ and $X$ have the same Hodge and Betti numbers.
\end{thm}
\begin{proof}
    The Betti numbers agree by construction, since the manifolds are diffeomorphic.
As for the Hodge numbers, the equality follows from Corollary~\ref{p00q}, 
Corollary~\ref{HodgeBetti} and Theorem~\ref{Hodge numbers21}, recalling that the Hodge numbers  
of an \(n\)-dimensional primary Hopf manifold (see \cite{mall1991cohomology}) are given by:
\[
h^{p,q}_{\bar{\partial}} =
\begin{cases}
1, & \text{if } (p,q) \in \{(0,0),(0,1),(n,n-1),(n,n)\}\\[2mm]
0, & \text{otherwise.}
\end{cases}
\]
\end{proof}

We are now ready to prove Theorem~\ref{Hodge decomposition}.
First, we observe that Theorem~\ref{Hodge numbers21} implies

\begin{equation}
\begin{split}
    h^{0,1}(X) + h^{1,0}(X) &= 1 = b_1(X),\\
    h^{n,n-1}(X) + h^{n-1,n}(X) &= 1 = b_{2n-1}(X).
\end{split} 
\end{equation}
Since \(\widehat{\mathbb{CP}^n}\) is a proper modification of the projective space \(\mathbb{CP}^n\), it is Moishezon \cite[Theorem 2.2.14]{Moishezon}. 
The Hodge decomposition holds on Moishezon manifolds \cite[Theorem 2.2.18]{Moishezon}.
By Theorem~\ref{Hodge numbers21}, we can conclude the proof.
Indeed, if \(2\le k\le 2n-2 \), we obtain

\begin{equation}
\begin{split}
   \sum_{p+q=k} h^{p,q}(X)&\overset{\ref{Hodge numbers21}}{=}\sum_{p+q=k}h^{p,q}(\widehat{\mathbb{CP}^n})-h^{p,q}(\mathbb{CP}^n) \\
  &=b_k(\widehat{\mathbb{CP}^n})-b_k(\mathbb{CP}^n)\overset{\ref{Hodge numbers21}}{=}b_k(X).
\end{split} 
\end{equation}

\begin{rem}\label{lckremark}
   As a consequence of Corollary~\ref{HodgeBetti} and of the previous discussion, 
we observe that if \(X\) is a compact (lcK) manifold and \(\widehat{X}\) is a proper holomorphic modification 
of \(X\) at a single point, then \(\widehat{X}\) satisfies the Hodge decomposition if and only if 
\(X\) does.

\end{rem}

\section{Examples}\label{SectionExamples}
In this section, we explicitly compute the Hodge and Betti numbers for two
classes of Kato manifolds. The first class concerns Kato manifolds obtained
through successive blow-ups at points: the modification \(\pi\) is a finite
sequence of point blow-ups, starting with the blow-up at the origin in the ball
\(\mathbb{B}\), and at each step, blowing up a point on the exceptional divisor
created in the previous step. This family of Kato manifolds admits a locally
conformally Kähler metric
(see \cite{Bru11, istrati2019classkatomanifolds}).
The second class is that of toric Kato manifolds, arising from toric modifications and studied in \cite{IOPR22}.

\begin{ex}[Blow-ups at points]
   Let $X = X(\pi,\sigma)$ be a Kato manifold such that the Kato modification
$\pi$ is a composition of $r$ blow-ups at points. Iterating Theorem~\ref{blowupd}, the Betti numbers of \(X\) are given by:
\begin{align*}
    b_k(X)=b_k(\widehat{H}) &= \begin{cases}
        1, & k \in \{0,1,2n-1,2n\}\\
        r, & k=2j,1\le j\le n-1\\
        0, & \text{otherwise},
    \end{cases}
\end{align*}
where \(\widehat{H}\) is the modification of a primary Hopf manifold $H$ obtained via small deformation from \(X\).
Applying Theorem~\ref{Hodge numbers} and Theorem~\ref{blowupD} iteratively, we find that the Hodge numbers of \(X\) are given by:
\begin{align*}
    h^{p,q}(X) &= \begin{cases}
        1, & (p,q) \in \{(0,0),(0,1),(n,n-1),(n,n)\}\\
        r, & 1\le p = q\le n-1 \\
        0, & \text{otherwise}.
    \end{cases}
\end{align*}
In particular, we can directly see that the Hodge decomposition holds for \(X\) without relying on the arguments in Section~\ref{finalsection}.

\begin{rem}
    In this case, \(h^{1,2}(X) = 0\); and hence, by \cite[Theorem 1.6]{IOPR22}, Kato manifolds arising from blow-ups at points cannot admit pluriclosed metrics in dimensions greater than two.  
Given that they admit lcK metrics, this provides additional evidence for the folklore conjecture asserting that compact lcK manifolds admit pluriclosed metrics only in dimension two.\\
We note that Kato manifolds with non-vanishing $h^{1,2}$ do exist.
Indeed, in dimension $n \ge 4$, one may consider the Kato manifold
$X = X(\pi, \sigma)$, where $\pi$ is a modification obtained as the composition
of two blow-ups: first, the blow-up of the unit ball in $\mathbb{C}^n$ at the origin,
with exceptional divisor $E = \mathbb{CP}^{n-1}$; and second, the blow-up
along an elliptic curve $Z \subset E$.
By Theorem~\ref{Hodge numbers21} together with the blow-up formula,
Theorem~\ref{blowupD}, we obtain
\[
h^{1,2}(X) = h^{0,1}(Z) = 1.
\]
Furthermore, $X$ is lcK, since the modification $\pi$ is K\"ahler
\cite[Theorem~10.3]{IOPR22}.

\end{rem}

\end{ex}

\begin{ex}[Toric Kato manifolds]
    In \cite{IOPR22}, the authors study a class of toric Kato manifolds. We briefly recall the construction. Let \(N := \mathbb{Z}^n\) and denote by \(T_N\) the complex torus
\[
T_N := N \otimes_{\mathbb{Z}} \mathbb{C}^* = (\mathbb{C}^*)^n .
\]
A proper modification \(\pi : \widehat{\mathbb{C}^n} \to \mathbb{C}^n\) at \(0 \in \mathbb{C}^n\) 
is called a \emph{toric modification} if \(\widehat{\mathbb{C}^n}\) is a \(T_N\)-toric variety and 
\(\pi\) is \(T_N\)-equivariant.
A Kato data \((\pi,\sigma)\) is called a \emph{toric Kato data} if the extension 
\(\pi : \widehat{\mathbb{C}^n} \to \mathbb{C}^n\) is a smooth toric modification at the origin and there exists 
\(\nu \in \mathrm{Aut}_{\mathrm{gr}}(T_N)\) such that \(\sigma(\underline{\lambda} z)=\nu(\underline{\lambda})\,\sigma(z)\) for any \(z\in \mathbb{B}\) and \( \underline{\lambda}\in T_N\) for which this is defined.
In this case, we say that \(\sigma\) is \(\nu\)-equivariant.  
A Kato manifold is called a \emph{toric Kato manifold} if it admits a toric Kato data.

The Betti numbers of toric Kato manifolds are known.  
If \(X\) is an \(n\)-dimensional toric Kato manifold, then its Betti numbers 
\cite[Theorem~6.1]{IOPR22} are given by
\[
b_k(X)=
\begin{cases}
   1, \quad &k\in\{0,1,2n-1,2n\}\\
   -1+\sum_{s=j}^{n}(-1)^{\,s-j}\binom{s}{j}\left(a_{n-s}+\binom{n}{s+1}\right),& k=2j, 1\le j\le n-1 \\
   0,& \text{otherwise},
\end{cases}
\]
where \(a_k\) denotes the number of \(k\)-dimensional cones in the fan defining \(\widehat{\mathbb{C}^n}\).
In particular, the Betti numbers of $\widehat{\mathbb{CP}^n}$ are given by: 
\[b_k(\widehat{\mathbb{CP}^n})=
\begin{cases}
    \sum_{s=j}^{n}(-1)^{\,s-j}\binom{s}{j}\left(a_{n-s}+\binom{n}{s+1}\right), \quad &k=2j, 0\le j\le n\\
    0,&\text{otherwise}.
\end{cases}
\]
The Hodge numbers of $\widehat{\mathbb{CP}^n}$ satisfy
$h^{p,q}(\widehat{\mathbb{CP}^n}) = 0$ whenever $p$ is different from $q$
\cite[Corollary~12.7]{Danilov}.  
Without appealing to the Moishezon property, the Hodge decomposition for 
$\widehat{\mathbb{CP}^n}$ follows from the fact that it holds for any complete smooth toric variety 
\cite[Theorem~12.5]{Danilov}.  
Consequently, for any $p$ we have
\[
h^{p,p}(\widehat{\mathbb{CP}^n}) = b_{2p}(\widehat{\mathbb{CP}^n}).
\]
Using Theorem~\ref{Hodge numbers21}, we can compute the Hodge numbers of
toric Kato manifolds directly from the fan defining
$\widehat{\mathbb{C}^n}$.
\begin{prop}
    Let $X = X(\pi,\sigma)$ be a toric Kato manifold, and let $\pi:\widehat{\mathbb{C}^n}\to\mathbb{C}^n$ be the associated toric modification at the origin.  
    Then the Hodge numbers of $X$ are given by:
    \[
    h^{p,q}(X)=
    \begin{cases}
         1, & (p,q)\in\{(0,0),(0,1),(n,n-1),(n,n)\}\\
         -1+\displaystyle\sum_{s=p}^{n}(-1)^{\,s-p}\binom{s}{p}\!
         \left(a_{n-s}+\binom{n}{\,s+1\,}\right),
         & 1\le p=q\le n-1\\
         0, &\text{otherwise},
    \end{cases}
    \]
    where $a_k$ denotes the number of $k$-dimensional cones in the fan
    defining $\widehat{\mathbb{C}^n}$.
\end{prop}

\end{ex}

\section{Bott--Chern cohomology}
In this final section, we apply the same arguments used to compute the Dolbeault
cohomology of Kato manifolds to relate certain Bott--Chern numbers
of a Kato manifold $X$ to those of $\widehat{H}$. We also show that, in these
cases, Bott--Chern cohomology coincides with Dolbeault cohomology.
Before proceeding, we briefly recall some definitions and properties of the
cohomologies arising in the non-K\"ahler setting, such as Bott--Chern and
Aeppli cohomology. They build, on general complex manifolds, a bridge between Dolbeault and de~Rham cohomology and are naturally associated with the complex structure
of the manifold. Since the Dolbeault picture for Kato manifolds is now
clear, it is natural to investigate them as well.

Let $X$ be a complex manifold.
The Bott--Chern cohomology of $X$ is the bi-graded algebra
$$ H^{\bullet,\bullet}_{BC}(X) \;:=\; \frac{\ker\partial\cap\ker\overline{\partial}}{\text{Im }\partial\overline{\partial}}\;, $$
and the Aeppli cohomology of $X$ is the bi-graded $H^{\bullet,\bullet}_{BC}(X)$-module
$$ H^{\bullet,\bullet}_{A}(X) \;:=\; \frac{\ker\partial\overline{\partial}}{\text{Im }\partial+\text{Im }\overline{\partial}} \;.$$
The identity map induces the following natural morphisms of (bi-)graded $\mathbb{C}$-vector spaces:
$$
\xymatrix{
 & H^{\bullet,\bullet}_{BC}(X) \ar[d]\ar[ld]\ar[rd] & \\
 H^{\bullet,\bullet}_{\partial}(X) \ar[rd] & H^{\bullet}_{dR}(X;\mathbb{C}) \ar[d] & H^{\bullet,\bullet}_{\overline{\partial}}(X) \ar[ld] \\
 & H^{\bullet,\bullet}_{A}(X) &
}
$$
In general, the maps above are neither injective nor surjective.
They are all isomorphisms if and only if the manifold $X$
satisfies the $\partial\overline{\partial}$-lemma.
In particular, in the K\"ahler case all these cohomology groups coincide,
whereas in the non-K\"ahler setting they provide additional information
on the complex structure.
For a detailed treatment of these cohomologies, see~\cite{Angellabook}.

To any complex manifold $X$ and any bidegree $(p,q)$ one can associate the Schweitzer complex
$(\mathcal{L}^{\bullet}_{p,q},d_{\mathcal{L}})$ (see~\cite{schweitzer2007autourlacohomologiebottchern,StelzigJonas2025SROT}). It is a simple complex of locally free sheaves defined as follows:
\begin{align*}
\mathcal{L}_{p,q}^k&:=\bigoplus_{\substack{r+s=k\\r<p,s<q}} \mathcal{A}^{r,s}_X&\text{if }k\leq p+q-2,\\
\mathcal{L}_{p,q}^k&:=\bigoplus_{\substack{r+s=k+1\\r\geq p,s\geq q}} \mathcal{A}^{r,s}_X&\text{if }k\geq p+q-1,
\end{align*}
where $\mathcal{A}_X^{r,s}$ is the sheaf of smooth complex valued $(r,s)$-forms of $X$.
The differential $d_{\mathcal{L}}$ is given by
\[
\cdots\overset{\text{pr}\,\circ d}{\longrightarrow}\mathcal{L}_{p,q}^{p+q-3}\overset{\text{pr}\,\circ d}{\longrightarrow}\mathcal{L}_{p,q}^{p+q-2}\overset{\partial\overline{\partial}}{\longrightarrow}\mathcal{L}_{p,q}^{p+q-1}\overset{d}{\longrightarrow}\mathcal{L}_{p,q}^{p+q}\overset{d}{\longrightarrow}\cdots\: ,
\]
where $\text{pr}$ denotes the projection from the sheaf of all forms in a given degree
to the direct summand $\mathcal{L}_{p,q}^k$.
By construction, denoting by
$\mathbb{H}^{k}\big(\mathcal{L}_{p,q}^\bullet(X)\big)$
the $k$-th cohomology group associated with the Schweitzer complex $\mathcal{L}_{p,q}^\bullet$,
we have
\begin{equation}\label{S1}
    H^{p,q}_{BC}(X)\cong\mathbb{H}^{p+q-1}\big(\mathcal{L}_{p,q}^\bullet(X)\big),
\end{equation}
and
\begin{equation}\label{S2}
    H^{p-1,q-1}_{A}(X)\cong\mathbb{H}^{p+q-2}\big(\mathcal{L}_{p,q}^\bullet(X)\big).
\end{equation}

Given an open cover $\{U,V\}$ of $X$, the Mayer--Vietoris sequence associated
with the complex $\mathcal{L}^{\bullet}_{p,q}$ yields the exact sequence
\begin{align*}
    \mathbb{H}^{p+q-2}\!\big(\mathcal{L}_{p,q}(U \cap V)\big)
\to
\mathbb{H}^{p+q-1}\!\big(&\mathcal{L}_{p,q}(X)\big)
\to
\mathbb{H}^{p+q-1}\!\big(\mathcal{L}_{p,q}(U)\big)
\oplus
\mathbb{H}^{p+q-1}\!\big(\mathcal{L}_{p,q}(V)\big)\\
&\to
\mathbb{H}^{p+q-1}\!\big(\mathcal{L}_{p,q}(U \cap V)\big).
\end{align*}
Via the identifications (\ref{S1}) and (\ref{S2}), for any $p,q\ge 1$ this sequence becomes
\begin{equation}\label{MVBottAeppli}
H^{p-1,q-1}_A(U \cap V)
\to
H^{p,q}_{BC}(X)
\to
H^{p,q}_{BC}(U) \oplus H^{p,q}_{BC}(V)
\to
H^{p,q}_{BC}(U \cap V).
\end{equation}

We are now in a position to state and prove a relation between some Bott--Chern
numbers of a Kato manifold $X$ and those of $\widehat{H}$.

\begin{prop}\label{BottChern}
Let $X$ be a Kato manifold, and let $\widehat{H}$ be the modification of a primary Hopf manifold $H$ obtained via small deformation from $X$.
Then $X$ and $\widehat{H}$ have the same Bott--Chern numbers $h^{p,q}_{BC}$ in the following cases:
\begin{enumerate}
\item $p,q \ge 2$ and $p+q \le n-1$;
\item $p = n$ or $q = n$;
\item $p = 0$ or $q = 0$.
\end{enumerate}
\end{prop}

We begin by proving an analogue of Lemma~\ref{annulus} for the Aeppli and
Bott--Chern cohomology of an annulus.
Before doing so, we recall two results from~\cite{AngellaDaniele2013Bcaq} concerning Aeppli
and Bott--Chern cohomology.
We then state Proposition~\ref{generalization} and Corollary~\ref{HodgeBetti} for Bott--Chern cohomology, 
whose proofs are the same as in the Dolbeault cohomology case, but using Lemma~\ref{annulusBC} and the exact sequence in (\ref{MVBottAeppli}).

\begin{thm}[{\cite[Theorem~1.2]{AngellaDaniele2013Bcaq}}]\label{Angella1}
    Let $X$ be a complex manifold. Fix $(p,q)\in\big(\mathbb{N}\setminus\{0\}\big)^2$.
If
\[
H^{p,q}_{\overline{\partial}}(X)=\{0\}
\quad\text{and}\quad
H^{q,p}_{\overline{\partial}}(X)=\{0\},
\]
then
\[
H^{p,q}_{A}(X)=\{0\}.
\]
\end{thm}

\begin{thm}[{\cite[Theorem~1.1]{AngellaDaniele2013Bcaq}}]\label{Angella2}
    Let $X$ be a complex manifold. Fix $(p,q)\in\big(\mathbb{N}\setminus\{0\}\big)^2$.

\medskip
\noindent
 If
\[
\sum_{\substack{r+s=p+q-1 \\ s \ge \min\{p,q\}}}
\dim_{\mathbb{C}} H^{r,s}_{\overline{\partial}}(X) = 0,
\]
then there is a natural injective map
\[
H^{p,q}_{BC}(X) \longrightarrow H^{p+q}_{dR}(X;\mathbb{C}).
\]

\end{thm}

\begin{lem}\label{annulusBC}
	    Let $\mathcal{A}$ be an annulus in $\mathbb{C}^n$, with $n \ge 3$. Then the following vanishing results hold for the Aeppli and Bott--Chern cohomology of $\mathcal{A}$:
\begin{align*}
H^{p,q}_A(\mathcal{A}) &= 0, \quad  1 \le p,q \le n-2\\
H^{p,q}_{BC}(\mathcal{A}) &= 0, \quad  1 \le p,q \text{  and  } p+q\le n-1.
\end{align*}
	\end{lem}
\begin{proof}
    Both vanishing results follow from the vanishing of the Dolbeault cohomology of an annulus in bidegree $(p,q)$ for $1 \le q \le n-2$ (Lemma~\ref{annulus}), together with Theorems \ref{Angella1},\ref{Angella2}.
\end{proof}

\begin{prop}\label{generalization2}
Let $X$ be a compact complex manifold, and let 
$\pi : \widehat{X} \to X$ be a proper holomorphic modification of $X$ at a point 
$p \in X$.  
Let $\mathbb{B}$ be a ball in $X$ centered at $p$, and let 
$\pi : \widehat{\mathbb{B}} \to \mathbb{B}$ denote the induced modification on \(\mathbb{B}\).  
Then, for any $2 \le p,q \le n-1$ with $p+q\le n-1$, we have
\[
h^{p,q}_{BC}(\widehat{X})
    =h^{p,q}_{BC}(X)+ h^{p,q}_{BC}(\widehat{\mathbb{B}}).
\]

\end{prop}
As a direct application of Proposition~\ref{generalization2}, we obtain the following corollary:

\begin{cor}\label{BC22}
    Let $X$ be a compact complex manifold, and let 
$\pi : \widehat{X} \to X$ be a proper holomorphic modification of $X$ at a point 
$p \in X$.  
Let $\widehat{\mathbb{CP}^n}$ be the modification at a point of $\mathbb{CP}^n$ induced by $\pi$.
Then, for any $2 \le p,q \le n-1$ with $p+q\le n-1$, we have
\[
h^{p,q}_{BC}(\widehat{X})
    = h^{p,q}_{BC}(X)
      + h^{p,q}_{BC}(\widehat{\mathbb{CP}^n})-h^{p,q}(\mathbb{CP}^n).
\]
\end{cor}

Now we prove Proposition \ref{BottChern}.
\begin{proof}
Let $X$ be a Kato manifold with an open cover $(U,V)$ as in Section~\ref{picture}. Using the same arguments as in Section \ref{picture} and Proposition \ref{generalization2}, we find that, for any $2 \le p,q \le n-1$ with $p+q\le n-1$,
\begin{align*}
    H^{p,q}_{BC}(X) &\cong H^{p,q}_{BC}(V) \\
    &\cong H^{p,q}_{BC}(\widehat{\mathbb{B}}) \\
    &\cong H^{p,q}_{BC}(H) \oplus H^{p,q}_{BC}(\widehat{\mathbb{B}}) \cong H^{p,q}_{BC}(\widehat{H}),
\end{align*}
since the Bott--Chern numbers of a primary Hopf manifold are given by \cite[Theorem 4.2]{istrati2025propertieshopfmanifolds}:
      \begin{equation}\label{BCHopf}
         h_{BC}^{p,q}=\begin{cases}
      1, & \text{if}\,\, (p, q) \in \{(0, 0), (1,1), (n-1,n), (n, n-1), (n,n)\}\\
      0, & \text{otherwise.}
    \end{cases}      
      \end{equation} 
If $k \neq 0$, then
\[
h^{k,0}_{BC}(X) = 0 = h^{k,0}_{BC}(\widehat{H}),
\]
since both $X$ and $\widehat{H}$ satisfy $h^{k,0}_{\bar\partial}=0$.\\
If $k = 0$, then
\(
h^{0,0}_{BC}(X) = 1 = h^{0,0}_{BC}(\widehat{H}),
\)
as both manifolds are compact and connected.

We now study the Bott--Chern cohomology in bidegree $(p,q)$ with $p=n$. Assume first that $q \neq n,n-1$. Let $\alpha \in A^{n,q}(X)$ be a $\overline{\partial}$-closed $(n,q)$-form. Then $[\alpha] \in H^{n,q}_{\overline{\partial}}(X) \cong H^{0,n-q}_{\overline{\partial}}(X) = 0$, so that $\alpha = \overline{\partial}\gamma$ for some $\gamma \in A^{n,q-1}(X)$ $\partial$-closed. Moreover,
\[
H^{n,q-1}_{\partial}(X) = \overline{H^{q-1,n}_{\overline{\partial}}(X)} = \overline{H^{n-q+1,0}_{\overline{\partial}}(X)} = 0,
\]
which implies that $\gamma$ is $\partial$-exact. Hence, $\alpha$ is $\partial\overline{\partial}$-exact.
As this argument depends only on the Dolbeault cohomology of $X$, we obtain
\[
h^{n,q}_{BC}(X) = h^{n,q}_{BC}(\widehat{H}) = 0.
\]
Using an analogous argument for $\bar\partial$-closed $(n,q)$-forms not representing the generator of the corresponding Dolbeault cohomology group, we find that if $q = n-1$ or $q = n$, then
\[
h^{n,q}_{BC}(X) =1 = h^{n,q}_{BC}(\widehat{H}) .
\]
\end{proof}

We conclude with two theorems that establish a relation between the Bott--Chern
and Aeppli cohomology of Kato manifolds and the
corresponding Dolbeault cohomology in certain degrees.
\begin{thm}\label{BCsameH}
Let $X = X(\pi,\sigma)$ be a Kato manifold.
If $(p,q)$ satisfies either
\begin{enumerate}
    \item $p,q \ge 2$ and $p+q \le n-1$, or
    \item $p = n$, or
    \item $q = 0$,
\end{enumerate}
then the Bott--Chern numbers of $X$ coincide with the corresponding Dolbeault numbers:
\[
    h^{p,q}_{BC}(X) = h^{p,q}_{\overline{\partial}}(X).
\]
\end{thm}
\begin{proof}
    The manifold $\widehat{\mathbb{CP}^n}$ satisfies the $\partial\overline{\partial}$-lemma,
since it is Moishezon \cite{DGMS75}.
Consequently, the Dolbeault cohomology of $\widehat{\mathbb{CP}^n}$ coincides with its Bott--Chern cohomology. Therefore, for any $(p,q)$ such that $p,q \ge 2$ and $p+q \le n-1$, the claim follows from Theorem \ref{Hodge numbers21}, Proposition \ref{BottChern}, and Corollary \ref{BC22}: 
\begin{align*}
     h^{p,q}_{BC}(X)&\overset{\ref{BottChern}}{=}h^{p,q}_{BC}(\widehat{H})\\
     &\overset{\ref{BC22}}{=}h^{p,q}_{BC}(\widehat{\mathbb{CP}^n})-h^{p,q}_{BC}(\mathbb{CP}^n)+h^{p,q}_{BC}(H)\\
     &\overset{\ref{BCHopf}}{=}h^{p,q}_{\overline{\partial}}(\widehat{\mathbb{CP}^n})-h^{p,q}_{\overline{\partial}}(\mathbb{CP}^n)\overset{\ref{Hodge numbers21}}{=}h^{p,q}_{\overline{\partial}}(X).
\end{align*}
If $p = n$ or $q = 0$, the identity can be checked directly. Indeed, in this case the Bott--Chern numbers were computed in the proof of Proposition~\ref{BottChern}: 
they vanish except for $(p,q) \in \{(0,0),(n,n-1),(n,n)\}$, 
where they are equal to one.

\end{proof}

Using the duality between Bott--Chern and Aeppli cohomology~\cite{schweitzer2007autourlacohomologiebottchern},
\[
h^{p,q}_{BC}(X) \cong h^{n-p,n-q}_{A}(X),
\]
together with Serre duality for Dolbeault cohomology,
we obtain an analogue of Theorem~\ref{BCsameH} for Aeppli cohomology.
\begin{cor}
Let $X = X(\pi,\sigma)$ be a Kato manifold.
If $(p,q)$ satisfies either
\begin{enumerate}
    \item $p,q \le n-2$ and $p+q \ge n+1$, or
    \item $p = 0$, or
    \item $q = n$,
\end{enumerate}
then the Aeppli numbers of $X$ coincide with the corresponding Dolbeault numbers:
\[
    h^{p,q}_{A}(X) = h^{p,q}_{\overline{\partial}}(X).
\]
\end{cor}

\bibliographystyle{alpha}

\begin{thebibliography}{100}

\bibitem[AC13]{AngellaDaniele2013Bcaq} D. Angella, S. Calamai,  Bott--Chern cohomology and $q$-complete domains, {\it C. R. Math. Acad. Sci. Paris} {\bf 351} (2013), no.~9-10, 343--348; MR3072157

\bibitem[AG62]{AG62} A. Andreotti, H. Grauert,  Th\'eor\`eme de finitude pour la cohomologie des espaces complexes, {\it Bull. Soc. Math. France} {\bf 90} (1962), 193--259; MR0150342

\bibitem[AKMWo02]{AKMWo02} D. Abramovich, K. Karu, K. Matsuki, J. W\l odarczyk, Torification and factorization of birational maps, {\it J. Amer. Math. Soc. }{\bf 15} (2002), no.~3, 531--572; MR1896232


\bibitem[Ang14]{Angellabook} D. Angella, {\it Cohomological aspects in complex non-K\"ahler geometry}, Lecture Notes in Mathematics, {\bf 2095}, Springer, Cham, 2014; MR3137419

\bibitem[BHPVdV04]{VandeVen} W.~P. Barth, K. Hulek, C. A. M. Peters, A. Van de Ven, {\it Compact complex surfaces}, second edition, 
Ergebnisse der Mathematik und ihrer Grenzgebiete. 3. Folge. A Series of Modern Surveys in Mathematics, {\bf 4}, Springer, Berlin, 2004; MR2030225

\bibitem[BO25]{barbaro2025calabiyaulocallyconformallykahler} G. Barbaro, A. Otiman, Calabi--Yau locally conformally K\"ahler manifolds, 2025, \url{https://arxiv.org/abs/2509.18364}.

\bibitem[Bru11]{Bru11} M. Brunella,  Locally conformally K\"ahler metrics on Kato surfaces, {\it Nagoya Math. J.} {\bf 202} (2011), 77--81; MR2804546

\bibitem[Dan78]{Danilov} V.~I. Danilov, The geometry of toric varieties, {\it Uspekhi Mat. Nauk} {\bf 33} (1978), no.~2(200), 85--134, 247; MR0495499

\bibitem[Dem12]{Demaillyag} J.-P. Demailly, {\it Complex Analytic and Differential Geometry}, 2012, \url{https://www-fourier.ujf-grenoble.fr/~demailly/manuscripts/agbook.pdf}.

\bibitem[DGMS75]{DGMS75} P. Deligne, P. Griffiths, J. Morgan, D. Sullivan, Real homotopy theory of K\"ahler manifolds, {\it Invent. Math.} {\bf 29} (1975), no.~3, 245--274; MR0382702

\bibitem[Dlo84]{Dlo84} G. Dloussky, Structure des surfaces de Kato, {\it M\'em. Soc. Math. France (N.S.)} No. {\bf 14} (1984), {\rm ii}+120 pp.; MR0763959

\bibitem[GR04]{CartanB} H. Grauert, R. Remmert, {\it Theory of Stein spaces}, translated from the German by Alan Huckleberry, 
Grundlehren der Mathematischen Wissenschaften, {\bf 236}, Springer, Berlin-New York, 1979; MR0580152

\bibitem[H90]{Hormander1990} L.~V. H\"ormander, {\it An introduction to complex analysis in several variables}, third edition, 
North-Holland Mathematical Library, {\bf 7}, North-Holland, Amsterdam, 1990; MR1045639

\bibitem[Huy05]{Huybrechts} D. Huybrechts, {\it Complex geometry}, Universitext, Springer, Berlin, 2005; MR2093043

\bibitem[IO25]{istrati2025propertieshopfmanifolds} N. Istrati, A. Otiman, On some properties of Hopf manifolds, in {\it Real and complex geometry---in honour of Paul Gauduchon}, 201--218, Springer, Cham, ; MR4957150

\bibitem[IOP21]{istrati2019classkatomanifolds} N. Istrati, A. Otiman, M. Pontecorvo, On a class of Kato manifolds, {\it Int. Math. Res. Not. IMRN} {\bf 2021}, no.~7, 5366--5412; MR4241131

\bibitem[IOPR22]{IOPR22} N. Istrati, A. Otiman, M. Pontecorvo, M. Ruggiero, Toric Kato manifolds, {\it J. \'Ec. polytech. Math.} {\bf 9} (2022), 1347--1395; MR4482304

\bibitem[Kat78]{Kato77} M. Kato, Compact complex manifolds containing ``global''\ spherical shells. I, in {\it Proceedings of the International Symposium on Algebraic Geometry (Kyoto Univ., Kyoto, 1977)}, pp. 45--84, Kinokuniya Book Store, Tokyo, 1978; MR0578853

\bibitem[Lau75]{Laufer} H.~B. Laufer, On the infinite dimensionality of the Dolbeault cohomology groups, {\it Proc. Amer. Math. Soc.} {\bf 52} (1975), 293--296; MR0379887

\bibitem[Mal91]{mall1991cohomology} D. Mall, The cohomology of line bundles on Hopf manifolds, {\it Osaka J. Math.} {\bf 28} (1991), no.~4, 999--1015; MR1152964

\bibitem[Mal96]{MallFred} D. Mall, Contractions, Fredholm operators and the cohomology of vector bundles on Hopf manifolds, {\it Arch. Math. (Basel)} {\bf 66} (1996), no.~1, 71--76; MR1363779

\bibitem[Men20]{Meng20} L. Meng, Leray-Hirsch theorem and blow-up formula for Dolbeault cohomology, {\it Ann. Mat. Pura Appl. (4)} {\bf 199} (2020), no.~5, 1997--2014; MR4142860

\bibitem[MM07]{Moishezon} X. Ma, G. Marinescu, {\it Holomorphic Morse inequalities and Bergman kernels}, Progress in Mathematics, {\bf 254}, Birkh\"auser, Basel, 2007; MR2339952

\bibitem[OT21]{OTHodge} A. Otiman, M. Toma, Hodge decomposition for Cousin groups and for Oeljeklaus--Toma manifolds, {\it Ann. Sc. Norm. Super. Pisa Cl. Sci. (5)} {\bf 22} (2021), no.~2, 485--503; MR4288663

\bibitem[OV24]{ornea2024principleslocallyconformallykahler} L. Ornea, M. Verbitsky, {\it Principles of locally conformally K\"ahler geometry}, Progress in Mathematics, {\bf 354}, Birkh\"auser/Springer, Cham, [2024] \copyright 2024; MR4771164

\bibitem[OVV13]{orneablowupslcK} L. Ornea, M. Verbitsky, V. Vuletescu, Blow-ups of locally conformally K\"ahler manifolds, {\it Int. Math. Res. Not. IMRN}  2013, {\bf 12}, 2809--2821; MR3071665

\bibitem[RYY19]{rao2018dolbeaultcohomologiesblowingcomplex} S. Rao, S. Yang, X. Yang, Dolbeault cohomologies of blowing up complex manifolds, {\it J. Math. Pures Appl. (9)} {\bf 130} (2019), 68--92; MR4001628

\bibitem[Sch07]{schweitzer2007autourlacohomologiebottchern} M. Schweitzer, Autour de la cohomologie de Bott--Chern, 2007, \url{https://arxiv.org/abs/0709.3528}.

\bibitem[Sha92]{Shabat1992} B.V. Shabat, {\it Introduction to complex analysis. Part II}, translated from the third (1985) Russian edition by J. S. Joel, 
Translations of Mathematical Monographs, {\bf 110}, Amer. Math. Soc., Providence, RI, 1992; MR1192135

\bibitem[Ste25]{StelzigJonas2025SROT} J. Stelzig, Some remarks on the Schweitzer complex, {\it Ann. Inst. Fourier (Grenoble)} {\bf 75} (2025), no.~1, 35--47; MR4862341

\bibitem[Tsu94]{Tsukada} K. Tsukada, Holomorphic forms and holomorphic vector fields on compact generalized Hopf manifolds, {\it Compositio Math.} {\bf 93} (1994), no.~1, 1--22; MR1286795

\bibitem[Vai80]{Vaisman} I. Vaisman, On locally and globally conformal K\"ahler manifolds, {\it Trans. Amer. Math. Soc.} {\bf 262} (1980), no.~2, 533--542; MR0586733

\bibitem[Voi02]{Voisin2002} C. Voisin, {\it Hodge theory and complex algebraic geometry. I}, translated from the French original by Leila Schneps, 
Cambridge Studies in Advanced Mathematics, {\bf 76}, Cambridge Univ. Press, Cambridge, 2002; MR1967689

\end{thebibliography}

\end{document}